\documentclass[openright]{report}
\usepackage{pdfpages}
\usepackage{hujititlepage}

\usepackage{amsmath}
\usepackage{color}
\usepackage{amsfonts}
\usepackage{amsthm}
\usepackage{hyperref}
\usepackage{amssymb}
\usepackage{bbm}
\usepackage{graphicx}

\usepackage{algorithm}
\usepackage{algorithmic}

\newtheorem{lemma}{Lemma}
\newtheorem{corollary}{Corollary}
\newtheorem{theorem}{Theorem}

\newtheorem*{conjecture*}{Conjecture}

\newtheorem{remark}{Remark}

\hypersetup{linktocpage}

{
\vspace{0.01cm}
\begin{center}
	\begin{boxedminipage}{0.8\linewidth}
		\begin{center}
			\textbf{\texttt{#1}}
		\end{center}
	} %
	{
	\end{boxedminipage}
\end{center}
 \vspace{0.3cm}
}

\newcommand{\bx}{{\mathbf x}}
\newcommand{\by}{{\mathbf y}}
\newcommand{\bz}{{\mathbf z}}
\newcommand{\bw}{{\mathbf w}}
\newcommand{\bb}{{\mathbf  b}}
\newcommand{\bv}{{\mathbf v}}
\newcommand{\bu}{{\mathbf u}}

\newcommand{\be}{{\mathbf e}}
\newcommand{\br}{{\mathbf{r}}}

\newcommand{\balpha}{\boldsymbol{\alpha}}

\newcommand{\bE}{\mathbb{E}}
\newcommand{\reals}{\mathbb{R}}

\newcommand{\bN}{\mathbb{N}}

\newcommand{\bC}{\mathbb{C}}

\newcommand{\spec}[1]{\sigma\!\circpar{#1}}

\newcommand{\eqdef}{\stackrel{\vartriangle}{=}}

\newcommand{\cA}{\mathcal{A}}
\newcommand{\cB}{\mathcal{B}}
\newcommand{\cC}{\mathcal{C}}
\newcommand{\cD}{\mathcal{D}}

\newcommand{\cF}{\mathcal{F}}
\newcommand{\cG}{\mathcal{G}}

\newcommand{\cI}{\mathcal{I}}
\newcommand{\cJ}{\mathcal{J}}

\newcommand{\cM}{\mathcal{M}}

\newcommand{\cP}{\mathcal{P}}

\newcommand{\cS}{\mathcal{S}}

\newcommand{\cU}{\mathcal{U}}

\renewenvironment{proof}{\par\noindent{\bf Proof\ }}{\hfill\BlackBox\\[2mm]}
\newcommand{\BlackBox}{\rule{1.5ex}{1.5ex}}

\makeatletter
\def\moverlay{\mathpalette\mov@rlay}
\def\mov@rlay#1#2{\leavevmode\vtop{%
   \baselineskip\z@skip \lineskiplimit-\maxdimen
   \ialign{\hfil$\m@th#1##$\hfil\cr#2\crcr}}}
\newcommand{\charfusion}[3][\mathord]{
    #1{\ifx#1\mathop\vphantom{#2}\fi
        \mathpalette\mov@rlay{#2\cr#3}
      }
    \ifx#1\mathop\expandafter\displaylimits\fi}
\makeatother

\DeclareMathOperator*{\argmin}{argmin} 

\renewcommand{\eqref}[1]{Equation~(\ref{#1})}
\newcommand{\ineqref}[1]{Inequality~(\ref{#1})}

\newcommand{\thmref}[1]{Theorem~\ref{#1}}
\newcommand{\lemref}[1]{Lemma~\ref{#1}}
\newcommand{\defref}[1]{Definition~\ref{#1}}
\newcommand{\corref}[1]{Corollary~\ref{#1}}

\newcommand{\coursename}{(67577) INTRODUCTION TO MACHINE LEARNING}

\newcommand{\handout}[5]{
   \renewcommand{\thepage}{#1-\arabic{page}}
   \noindent
   \begin{center}
   \framebox{
      \vbox{
    \hbox to 5.78in { {\bf \coursename}
         \hfill #2 }
       \vspace{4mm}
       \hbox to 5.78in { {\Large \hfill #5  \hfill} }
       \vspace{2mm}
       \hbox to 5.78in { {\it #3 \hfill #4} }
      }
   }
   \end{center}
   \vspace*{4mm}
}

\newcommand{\norm}[1]{\left\Vert#1\right\Vert}
\newcommand{\normsq}[1]{\left\Vert#1\right\Vert^2}
\newcommand{\inprod}[2]{ \left< #1 , #2 \right>}
\newcommand{\circpar}[1]{\left( #1 \right)}
\newcommand{\absval}[1]{\left|\,#1\,\right|}

\newcommand{\set}[1]{\left\{ #1 \right\} }
\newcommand{\myset}[2]{\left\{ #1 \right.\left|~ #2 \right\} }

\newcommand{\Diag}[1]{\text{Diag}\circpar{#1}}
\newcommand{\diag}[1]{\text{diag}\circpar{#1}}
\newcommand{\posdef}[2]{\cS^{#1}_{#2}}
\newcommand{\posdefun}[2]{\cF^{#1}_{#2}}
\newcommand{\mymat}[1]{\circpar{\begin{array}{cccccccc} #1 \end{array}}}

\newcommand{\bigO}[1]{\mathcal{O}{\left(#1\right)}}
\newcommand{\bigtO}[1]{\tilde{\mathcal{O}}{\left(#1\right)}}

\newcommand{\rhoM}{\rho_{\cM}}
\newcommand{\chiM}{\chi_{\cM}}

\newcommand{\IC}{\cI\cC}
\newcommand{\ICA}{\IC_\Sigma\circpar{\cA,\epsilon}}

\begin{document}
\title{On Lower and Upper Bounds in Smooth Strongly Convex Optimization\\ \vspace{2 mm} \LARGE{A Unified Approach via Linear Iterative Methods}}
\author{Yossi Arjevani}
\date{August 2014}
\maketitle


\begin{abstract}
In this thesis we develop a novel framework to study smooth and strongly convex optimization algorithms, both deterministic and stochastic. Focusing on quadratic functions we are able to examine optimization algorithms as a recursive application of linear operators. This, in turn, reveals a powerful connection between a class of optimization algorithms and the analytic theory of polynomials whereby new lower and upper bounds are derived. In particular, we present a new and natural derivation of Nesterov's well-known Accelerated Gradient Descent method by employing simple 'economic' polynomials. This rather natural interpretation of AGD contrasts with earlier ones which lacked a simple, yet solid, motivation. Lastly, whereas existing lower bounds are only valid when the dimensionality scales with the number of iterations, our lower bound holds in the natural regime where the dimensionality is fixed. \\
\\
Our main contributions can be summarized as follows: 
\begin{itemize}
	\item We define a class of algorithms ($p$-CLI) in terms of linear operations on the last $p$ iterations, and show that they subsume some of the most interesting algorithms used in practice.
	\item We prove that any $p$-CLI optimization algorithm must use at least
	\begin{align*}
		\tilde{\Omega}\circpar{\sqrt[p]{Q}\ln(1/\epsilon)} 
	\end{align*}
	iterations in order to obtain an $\epsilon$-suboptimal solution. As mentioned earlier, unlike existing lower bounds, our bound holds for a fixed dimensionality.
	\item We show that there exist matching $p$-CLI optimization algorithms which attain the convergence rates  stated above for all $p$. Alas, for $p>2$ this requires an expensive calculation which renders this algorithm inapplicable in practice. 
		
	\item Consequently, we focus on a restricted easy-to-compute subclass of $p$-CLI optimization algorithms. This yields a new derivation of Full Gradient Descent, Accelerated Gradient Descent, The Heavy-Ball method and potentially others, by following a very natural line of argument. 
	

	\item We present new schemes which offer better utilization of second-order spectral information by exploiting breaches in existing lower bounds. This leads to a new optimization algorithm which obtains a rate of $\sqrt[3]{Q}\ln(1/\epsilon)$
	in presence of huge spectral gaps.

	\item Lastly, we show that the convergence analysis of SDCA is tight and suggest a few extensions of this framework so as to allow one to prove that the convergence analysis of other stochastic algorithms e.g. SAG, SVRG, Accelerated SDCA, etc. is tight, as well.
\end{itemize}

\end{abstract}

\newpage

\section*{\large{Acknowledgments}}

This thesis is dedicated to my grandmother Zlicha Arjevani, whose wisdom and kindness are surpassed only by her modesty.\\

I would like to thank: my advisor Prof. Shai Shalev-Shwartz for guiding this research
with his outstanding insights; Ohad Shamir for sharing his enlightening thoughts; My roommates, Yoav Wald, Alon Gonen, Nir Rosenfield, Dan Rosenbaum, Alon Cohen (Editorial Remarks), Maya Elroy and Cobi Cario, for being both great friends and colleagues; and Ma'ayan Maliach for her unconditional companionship. \\

Finally, to my family, Yoni, Lea, Hila and Maor Arjevani. Thank you for all your
love and support throughout the years.

\tableofcontents

\chapter{Introduction}

\section{Motivation}

In the field of mathematical optimization one is interested in efficiently solving the following  minimization problem
\begin{align} \label{opt:gen_min}
\min_{x\in X} f(x)
\end{align}
where the \emph{Objective Function}  $f$ is some real-valued function defined over the \emph{Constraints Set} $X$. \\
Many core problems in the field of Computer Science, Economic, and Operations Research can be readily expressed in this form, thus rendering this minimization task exceedingly significant. That being said, in its full generality this problem is just too hard to solve or even to approximate. Consequently,  various structural assumptions on the objective function and the constraints set, along with better-suited optimization algorithms, has been proposed in recent decades so as to make this problem viable.\\

One such case is smooth strongly convex  functions (see definition in \ref{section:convex_ana}), where $X$ is some Hilbert space. 
A wide range of applications together with efficient solvers have made this family of problems very important. Naturally, an interesting question  arises: How fast can these kind of problems be solved? better said, what is the worst-case computational complexity of solving \ref{opt:gen_min} for smooth strongly-convex objective functions to a given degree of accuracy? Arguably, an even more important question is: which optimization algorithm can be proven to attain this bound?\footnote{Natural as these questions might look nowadays, matters were quite different only few decades ago. In his book 'Introduction to Optimization' which dates back to 87', the well-known computer scientist Polyak B.T devotes a whole section as to: 'Why Are Convergence Theorems Necessary?'  (See section 1.6.2 in \cite{polyak1987introduction}).}
Prior to answering these, otherwise ill-defined, questions, one must first refer as to the exact nature of the underlying computational model. \\

Although being a widely accepted computational model in the theoretical computer sciences, the Turing Machine Model presents many obstacles when analyzing optimization algorithms. In their seminal work \cite{nemirovskyproblem}, Nemirovsky and Yudin overcame parts of this issue by proposing the \emph{black box computational model}. In the black-box model one assumes that the information regarding the objective function is acquired by iteratively querying an \emph{oracle}. Furthermore, the algorithm is assumed to have an unlimited amount of computational resources. In a nutshell, in each round one may carry out whichever computational task it desires, as long as it does not use any other information other than what was gathered thus far\footnote{In a sense, this model is dual to the Turing Machine model where all the information regarding the parameters of the problem is available prior to the execution of the algorithm, but the available computational resources are limited in time and space.}.\\

Using the black-box model Nemirovsky and Yudin have managed to show that the complexity of obtaining a sub-optimal solution of some smooth strongly convex function is highly correlated with a single number, the corresponding \emph{Condition Number}, denoted by $Q$ (See \ref{section:convex_ana}). This is demonstrated by the following well-celebrated theorem. Let $\cA$ be an optimization algorithm for smooth strongly-convex functions which employs first-order information only, i.e., upon querying at some point $x$, $\cA$ receives $(f(x),\nabla f(x))$ from the oracle.  Then, there exists an $L$-smooth $\mu$-strongly convex function such that for any $\epsilon>0$, if the number of issued queries is smaller than half the dimension of the problem $d$, then it must be at least
\begin{align} \label{opt:sqrtlb}
\tilde{\Omega} \circpar{\sqrt{Q} \ln (1/\epsilon)}
\end{align}
in order for $\cA$ to obtain a point $\hat{x}$ such that 
\begin{align} \label{ineq:eps_subopt}
f(\hat{x})-\min_x f(x) < \epsilon
\end{align}\\

This result can be seen as a starting point for this work. The restricted validity of the lower bound to the first $O\!\circpar{d}$ iterations is not a mere artifact of the proof. Indeed, from informational point of view, the minimizer of smooth strongly convex quadratic function can be found exactly after $d+1$ iterations.  Together with the fact that this bound is attainable by the Conjugate Gradient Descent method (CGD, see \cite{polyak1987introduction}), it seems that we have reached a dead end. Nevertheless, as we will shortly see, CGD has a very rich mechanism, as opposed to many popular algorithms which have been designed for large scale optimization scenarios, e.g. a huge problem dimension, that admit a much simpler mechanism. In this monograph, we will consider structural lower bounds that hold asymptotically for a certain class of algorithms. In contrast to existing techniques, ours does not exploit the fact that in any point only a limited amount of information about the objective is known, but rather assumes a certain structure of the optimization algorithm under consideration. Consequently, our technique lacks the universality of information-based approaches such as in \ref{opt:sqrtlb} and its modern successors, \cite{agarwal2012information,guzman2013lower,raginsky2011information} to name a few. However, by providing lower bounds on a certain class of optimization algorithms, it reveals number of important principles which have to be taken into account when designing efficient optimization algorithms.\\

The second part of this work concerns a cornerstone optimization algorithm, namely, Accelerated Gradient Descent (AGD), which is closely related to the lower bound stated above. At the time, when the work of Nemirovsky and Yudin was published, it was known that in order for the Gradient Descent (GD) optimization algorithm to obtain an $\epsilon$ sub-optimal solution, as in \ineqref{ineq:eps_subopt}, it suffices to have
\begin{align*}
\bigO{Q \ln (1/\epsilon)}
\end{align*}
first-order queries. As can be expected, closing the gap between this upper bound and the lower bound presented in \ref{opt:sqrtlb} has intrigued many researchers in the field. Eventually, it was this line of inquiry that  led to the discovery of AGD by Nesterov (see \cite{nesterov1983method}), a slight modification of the standard GD algorithm, whose iteration  complexity is 
\begin{align*}
\bigO{\sqrt{Q} \ln (1/\epsilon)}
\end{align*}
Unfortunately, AGD lacks the strong geometrical intuition which accompanies many optimization algorithms, such as GD and the Heavy Ball method. Primarily based on sophisticated technical manipulations, its proof has motivated many researchers in the field to look for a more tangible derivation of AGD (e.g. \cite{beck2009fast,baes2009estimate,tseng2008accelerated,sutskever2013importance,allen2014novel}). In a way, this downside has rendered the generalization of AGD to different optimization scenarios, such as constrained optimization problems, a highly non-trivial task which up to the present time does not admit a complete satisfactory solution. \\
Remarkably, the tools developed in the first part of the work have uncovered a simple way for deriving AGD, as well as the Heavy Ball method. This derivation emphasizes the algebraic nature of these methods.\\

Before delving into more details, we would like to present an instructive case which demonstrates the main ideas used in this work. 

\section{Case Study - Stochastic Dual Coordinate Ascent} \label{section:sdca_case_study}

We present the well-known optimization algorithm of Stochastic Dual Coordinates Ascent (SDCA) designed for solving Regularized Loss Minimization (RLM) problems. These problems are of great use in the field of Machine Learning. Next, by applying SDCA on quadratic loss functions we reformulate it as a linear iterative method. This, in turn, enables us to derive a lower bound on its convergence rate. For a thorough analysis of SDCA the reader is referred to \cite{shalev2013stochastic}.\\
\\
A smooth-RLM problem is an optimization program which takes the form of
\begin{align} \label{opt:RLM}
\min_{\bw\in\reals^d}P(\bw) &\eqdef \frac{1}{n} \sum_{i=1}^n \phi_i(\bw^\top \bx_i) + \frac{\lambda}{2} \normsq{\bw}
\end{align}
where $\phi_i$ are $1/\gamma$-smooth and convex,  $\bx_1,\ldots,\bx_n$ are vectors in $\reals^d$ and $\lambda$ is a positive constant. 
SDCA minimizes an equivalent optimization problem 
\begin{align*}
\min_{\alpha\in\reals^n} D(\balpha) \eqdef \frac{1}{n}\sum_{i=1}^n \phi_i^\star(\balpha_i) +\frac{1}{2\lambda n^2} \normsq{\sum_{i=1}^n \balpha_i \bx_i}
\end{align*}
by repeatedly picking $z\sim \cU([n])$ uniformly and minimizing $D(\balpha)$ over the $z$'th coordinate. Note that $\phi^\star$ denotes the \emph{Fenchel conjugate} of $\phi$.
The latter optimization program is referred to as the \emph{Dual Problem}, while the problem presented in (\ref{opt:RLM}) is called the \emph{Primal Problem}.
A careful analysis shows that it is possible to convert a high quality solution of the dual problem into a high quality solution of the primal problem, resulting in a convergence rate of 
\begin{align*} \label{bigo:RLM_conv_rate}
\bigtO{\circpar{n+\frac{1}{\lambda \gamma}}\ln(1/\epsilon)}\\
\end{align*}
Assume that $\phi_i$ are non-negative, $\phi_i(0)\le 1$ and $\norm{\bx_i}\le 1$ for all $i$. In what follows, we show that this analysis is tight. \\

Let us define 
\begin{align*}
\phi_i(y) = y^2,\quad i=1,\dots,n
\end{align*}
and define $\bx_1=\bx_2=\cdots=\bx_n=\frac{1}{\sqrt{n}}\mathbbm{1}$. This yields
\begin{align} 
D(\balpha) &= \frac{1}{2}\balpha^\top \circpar{ \frac{1}{2n}I+ \frac{1}{\lambda n^2} \mathbbm{1}\mathbbm{1}^\top}\balpha
\end{align}
Obviously, the unique minimizer of $D(\balpha)$ is $\balpha^*\eqdef0$. \\

Given $i\in[n]$ and $\balpha\in\reals^n$ , one can verify that 
\begin{align}
\argmin_{\balpha' \in\reals} D(\balpha_1,\dots,\balpha_{i-1},\balpha',\balpha_{i+1},\dots,\balpha_{n}) = \frac{-2}{2+\lambda n} \sum_{j\neq i} \balpha_j 
\end{align}
Thus the next test point $\balpha^+$, generated by taking a $\balpha_i$-coordinate step, is linear transformation of the previous point,
\begin{align} \label{eq:rlm_costep}
\balpha^+=\circpar{I-\be_i \bu_i^\top}\balpha
\end{align}
Where 
\begin{align*}
\bu_i^\top &\eqdef \circpar{\frac{2}{2+\lambda n }, \dots, \frac{2}{2+\lambda n },\underbrace{1}_{i\text{'s entry}},\frac{2}{2+\lambda n },\ldots, \frac{2}{2+\lambda n } }
\end{align*}

Now, let $\balpha^k,~ k=1,\dots,K$ denote the $k$'th test point. These points are randomly generated by minimizing $D(\balpha)$ over the $z_i$'th coordinate in the $i$'th iteration, where $z_1,z_2,\dots,z_K\sim \mathcal{U}([n])$ is a sequence of $K$ uniform distributed i.i.d random variables. Applying \ref{eq:rlm_costep} over and over again for some initialization point $\balpha^0$ we have
\begin{align*}
\balpha^k&=\circpar{I-\be_{z_K}\bu_{z_K}^\top}\circpar{I-\be_{z_{K-1}}\bu_{z_{K-1}}^\top}\cdots\circpar{I-\be_{z_1}\bu_{z_1}^\top}\balpha^0
\end{align*}
To compute $\mathbb{E}[\balpha^K]$ note that by the i.i.d hypothesis and by the linearity of the expectation operator,
\begin{align}
\mathbb{E}\left[\balpha^K\right]&=\mathbb{E}\left[\circpar{I-\be_{z_K}\bu_{z_K}^\top}\circpar{I-\be_{z_{K-1}}\bu_{z_{K-1}}^\top}\cdots\circpar{I-\be_{z_1}\bu_{z_1}^\top}\balpha^0\right]\nonumber\\
&=\mathbb{E}\left[\circpar{I-\be_{z_K}\bu_{z_K}^\top}\right]\mathbb{E}\left[\circpar{I-\be_{z_{K-1}}\bu_{z_{K-1}}^\top}\right]\cdots\mathbb{E}\left[\circpar{I-\be_{z_1}\bu_{z_1}^\top}\right]\balpha^0\nonumber\\
&=\mathbb{E}\left[\circpar{I-\be_{z}\bu_{z}^\top}\right]^K\balpha^0 \label{kpoint}
\end{align}
Clearly, the convergence rate of this process is governed by the largest eigenvalue of $$E\eqdef \mathbb{E}\left[\circpar{I-\be_{z}\bu_{z}^\top}\right]$$By a straightforward calculation it can be shown that the eigenvalue of $E$, ordered by magnitude, are
\begin{align}
\underbrace{\frac{1}{2/\lambda + n},\dots , \frac{1 }{2/\lambda + n}}_{n-1 \text{ times}}, 1 - \frac{2 + \lambda}{2+\lambda n} 
 \end{align}
By choosing $\balpha^0$ to be the following normalized eigenvector corresponding to the largest eigenvalue,
$$\balpha^0=\circpar{\frac{1}{\sqrt{2}},-\frac{1}{\sqrt{2}},0,\dots,0}$$
and plugging it into (\eqref{kpoint}), we can now bound from below the distance of $\mathbb{E}[\balpha^K]$ to the optimal point $\balpha^*=0$,
\begin{align*}
\norm{\mathbb{E}\left[\balpha^K\right]-\balpha^*}
&=\norm{\mathbb{E}\left[\circpar{I-\be_{z}\bu_{z}^\top}\right]^K\balpha^0}\\
&=\circpar{1 - \frac{1}{2/\lambda+n}}^K\norm{\balpha^0}\\
&=\circpar{1 - \frac{2}{(4/\lambda+2n-1)+1}}^K\\
&\ge \circpar{\exp\circpar{\frac{-1}{2/\lambda+n-1}}}^K
\end{align*}
The last inequality is due to the following,
\begin{align*}
1-\frac{2}{x+1}\ge \exp\circpar{\frac{-2}{x-1}},\quad \forall x\ge1
\end{align*}
Therefore, we get that in order to obtain a solution whose distance form the optimal solution is less than $\epsilon>0$, the number of iterations must satisfy 
\begin{align*} 
K&\ge \circpar{2/\lambda+n-1}\ln\circpar{1/\epsilon} 
\end{align*}
Thus showing that, up to logarithmic factors, the analysis of the convergence rate of SDCA is tight. We remark that the chosen loss functions $\phi_i$ are $2$-smooth and that the the convergence rate stated in  \ref{bigo:RLM_conv_rate} is given in terms of $D(\balpha)$, i.e. it determines how many iteration are required in order to obtain $\balpha\in\reals^n$ such that $D(\balpha)-D(\balpha*)<\epsilon$ for a given $\epsilon>0$. \\

Motivated by this example, a major part of the work is devoted to formalizing and generalizing various aspects of the preceding arguments.

\section{Organization and Contributions}

We present contributions of the work in order of appearance.\\

In Chapter \ref{chapter:pcli_method} we introduce basic terminology and tools from the theory of linear iterative methods which will be used throughout this work. We focus on what we call $p$-CLI methods, a specific type of linear iterative methods around which this work revolves. We investigate the convergence properties of these methods and derive lower bounds as well as upper bounds on the convergence rate using the spectral radius of the corresponding linear operators. In spite of being 
an elementary result in Matrix Theory, this lower bound does not seem to appear in this form in standard literature.\\

Chapter \ref{chapter:pcli_algos} is mainly devoted for establishing a framework which generalizes the SDCA example shown above. In essence, we apply a given optimization algorithm on quadratic functions and then we explore the structure of the resulting algorithm. Note that since the close neighborhood of minimizers of smooth strongly convex functions can be efficiently approximated by quadratic functions, this technique is quite an intuitive one for our purposes. It turns out that inspecting various algorithms under this framework, provides an insightful fresh way of systemically differentiating one optimization algorithm from another, thereby revealing, otherwise subtle, useful distinctions. Lastly, a $p$-CLI method which originates from an optimization algorithm has more structure than a general $p$-CLI method. This observation is formalized in the last section of this chapter, where we derive an extremely useful characterization of $p$-CLI optimization algorithms, a relatively wide class of optimization algorithms which includes SDCA, Full Gradient Descent, Accelerated Gradient Descent and the Heavy Ball method to name a few.  It is by this characterization that we convert claims regarding convergence properties of optimization algorithms into claims regarding polynomials, and vice versa.\\

Loosely speaking, a $p$-CLI optimization algorithm is an optimization algorithm whose test points are generated by repeatedly applying the same linear transformation on the previous $p$ test points so as to produce a new test point. In chapter \ref{chapter:lower_bounds} we present a novel lower bound on the convergence rate of such optimization algorithms. In effect, we prove that for any $p$-CLI optimization algorithm whose inversion matrix is diagonal, there exists an initialization point under which the convergence rate of this algorithm is asymptotically bounded from below by,
\begin{align} \label{opt:lblb_conv_pcli}
\Omega\circpar{\frac{\sqrt[p]{Q}-1}{\sqrt[p]{Q}+1}}
\end{align}
In a certain sense ,this lower bound forms a complementary result for the lower bound presented \ref{opt:sqrtlb}, due to Nemirovsky and Yudin.\\
The proof of this result is carried out by showing that it is possible to derive a lower bound on the convergence rates of $p$-CLI optimization algorithms by obtaining a lower bound on the maximal modulus root of some polynomial. Although, the range of techniques regarding upper bound for the moduli of roots of polynomials is vast, e.g. \cite{marden1966geometry,rahman2002analytic,milovanovic1994topics,walsh1922location,milovanovic2000distribution,fell1980zeros}, to the best of our knowledge, there are a very few cases where lower bounds are considered (see \cite{higham2003bounds} \footnote{In fact, our interest in maximal modulus root of polynomials is a result of trying to bound the spectral radius of 'generalized' companion matrices. As far as we know, this topic is poorly covered in the literature, as well. Few of which may be found in \cite{wolkowicz1980bounds,zhong2008bounds,horne1997lower,huang2007improving}.}). Thus, we designate part of this chapter for developing new tools which allows us to estimate the maximal modulus root of polynomials.\\
In the last section we consider the assumptions under which the lower bound shown above is tight. More precisely, we develop a novel algorithm which given a quadratic function  
\begin{align*}
f(\bx) = \frac{1}{2}\bx^\top A \bx  + \bb^\top\bx +c
\end{align*}
and the spectrum of $A$, converges to the minimizer of $f$ in a rate that is dictated by
 \ref{opt:lblb_conv_pcli}, for some prescribed $p\in\bN$. Unquestionably, knowing the spectrum of $A$ is a very strong requirement. Indeed, the chapter is concluded with a conjecture on polynomials which, if proven, would imply that for any $p\in\bN$
the worst convergence rate of certain type of efficient $p$-CLI algorithms is,
\begin{align}
\Omega\circpar{\frac{\sqrt{Q}-1}{\sqrt{Q}+1}}
\end{align}\\

In the last chapter we seek for creating new optimization algorithms by employing tools developed in previous chapters. This approach shows that Accelerated Gradient Descent method is deeply rooted in the analytic theory of polynomials. We use the very same approach to derive other well-known optimization algorithms such as the Gradient Descent and the Heavy Ball methods.\\

We conclude this monograph by presenting what we believe to be valuable research directions for future work. 




\chapter{\texorpdfstring{$p$}{p}-CLI Methods} \label{chapter:pcli_method}
The following sections cover useful notions from the field of Iterative Methods and Matrix Theory, as well as important convergence properties of Linear Iterative Methods. These tools will be used throughout this monograph. 

\section{Definitions}
For our purposes, a \emph{$p$-step Iterative Method} $\cM$ is a specification of points$$\bx^0,\bx^1,\bx^2,\dots,\bx^{p-1}\in\reals^d$$ at which a solver for some task is initiated, and rules of update 
\begin{align} \label{def:iterative_method}
\bx^{k} = \phi_{k-1}\circpar{\bx^{k-1},\bx^{k-2},\dots,\bx^{k-p}} \quad k=p,p+1,\dots..
\end{align}
where $\{\phi_k\}_{k=p}^\infty$ is a sequence of transformations, each of which may be drawn randomly according to some distribution $\Phi_k$ over a set of transformations. Iterative methods whose $\Phi_k$ are identical is said to be \emph{Stationary} \footnote{Note that it is possible to extend the definition of $\Phi_k$ so that they are dependent on the entire history $\bx^{k-1},\bx^{k-2},\dots,\bx^0$, though we will not need this generality.}.

A key observation is that any $p$-step iterative method may be reduced to a single step iterative method by increasing the \emph{Dimension of the Problem}, which will be denoted by $d$. This can be done by introducing new variables in some possibly higher-dimensional Euclidean space $\reals^{pd}$
\begin{align*}
\bz^k = \mymat{\bx^{k}\\\bx^{k+1}\\\vdots\\\bx^{k+(p-1)}} \in \reals^{pd},\qquad k=0,1,2,\dots
\end{align*}
In which case, the iterative method may be re-specified in terms of $\reals^{pd}$, which we call the \emph{Optimization Space}, by
\begin{align*}
\bz^0 = \mymat{\bx^0\\\bx^1\\\vdots\\\bx^{p-1}} \in \reals^{pd},\qquad
\bz^k = \psi_{k-1}(\bz^{k-1}) 
\eqdef \mymat{\bx^{k} \\ \bx^{k+1} \\\vdots \\ \bx^{k+(p-2)}\\ \phi_{k+(p-2)}\circpar{\bx^{k+(p-2)},\dots , \bx^{k}, \bx^{k-1} }}
\end{align*}\\

An important kind of iterative methods is the \emph{Linear Multi-step iterative Methods}, for which all $\phi_k$ are affine transformations. That is, for all $k=1,2,\dots$ and $i\in[p]$, there exist $M_{i,k-1}\in\reals^{d\times d}$ and $\bv_{k-1}\in\reals^d$ such that 
\begin{align*}
\bx^k = \phi_{k-1}\circpar{\bx^{k-1}, \bx^{k-2},\dots,\bx^{k-p}} = \sum_{i=1}^p M_{i,k-1} \bx^{k-i}+ \bv_{k-1}
\end{align*}
Casting a $p$-step iterative method which is both stationary and linear as a single step one, yields a type of iterative method called \emph{$p$-Canonical Linear Iterative Method}, abbreviated $p$-CLI method\footnote{It is worth mentioning that this term is not conventional.}. This method may be specified by some appropriate random variables $C_0,\dots, C_{p-1}\in\reals^{d \times d}$ and $\bv\in\reals^{pd}$ as,
\begin{align} \label{def:pcli}
\mymat{\bx^0\\ \vdots\\ \bx^{p-1}} \in \reals^{pd}, \quad \mymat{\bx^{k}\\\bx^{k+1}\\\vdots\\\bx^{k+(p-1)}} =
\underbrace{\mymat{0_d & I_d &&& \\&0_d & I_d&&\\ &&&&\\ && \ddots&\ddots\\&&& \\&&& 0_d & I_d\\C_{0}&&\dots&C_{p-2}&C_{p-1} }}_{M} \mymat{\bx^{k-1}\\\bx^{k}\\\vdots\\\bx^{k+(p-2)}}+\bv
\end{align} 
Equivalently,
\begin{align}\label{Eq:canonical_dynamic}
\bz^0 \in \reals^{pd}, \quad \bz^k =M\bz^{k-1}+\bv
\end{align}
In which case $M$ is called the \emph{Iteration Matrix}, $~C_0,\dots, C_{p-1}$ the \emph{Coefficients Matrices} and $\bv$ the \emph{Free Summand}. In $p$-CLI methods we assume that in each iteration $M$ and $\bv$ are drawn randomly and independently of previous realizations.

Throughout this monograph we will be mainly interested in $p$-CLI iterative methods. The next sections explore the convergence properties of methods of this nature.

\section{Basic Convergence Properties} \label{section:cli_conv_prop}
As it turns out, the convergence properties of the sequence $\circpar{\bE \bz^k}_{k=0}^\infty$, generated by a $p$-CLI iterative method, are governed by the spectral content of the iteration matrix. Therefore we shall make a slight diversion in order to introduce few elementary notions from the Spectral Theory of Matrices. \\

Let $A\in\reals^{d\times d}$ be a square matrix. If
\begin{align*}
A\bx = \lambda \bx 
\end{align*}
for some $\bx\neq0\in\bC^d$  and $\lambda\in\bC$, then we say that $\bx$ is \emph{eigenvector} of $A$ and that $\lambda$ is an \emph{eigenvalue} of $A$. Together, the pair $(\lambda,\bx)$ forms an \emph{eigenpair}.\\
The \emph{Spectrum} of $A$
\begin{align*}
\spec{A} = \left\{ \lambda\in\bC \left|~ \exists \bx\neq0\in\bC^d, A\bx=\lambda \bx  \right. \right\}
\end{align*}
 is the set of all eigenvalues of $A$. Observe that 
$\lambda\in\spec{A}$ if and only if
\begin{align*}
 \exists \bx\neq0, A\bx=\lambda \bx \iff  \exists \bx\neq0, (A-\lambda I)\bx=0 \iff \det\absval{A-\lambda I}=0
\end{align*}
Thus, $\spec{A}$ is exactly all the roots of $\chi_A(\lambda)=\det\absval{A-\lambda I}$, the \emph{Characteristic Polynomial} of $A$. Denoting the set of all the roots of $\chi_A(\lambda)$ by $\spec{\chi_A}$, we may succinctly express the last consequence by $\spec{A}=\spec{\chi_A}$.
  
 An important characteristic of the spectrum of a matrix is its \emph{Spectral Radius} defined by
\begin{align*}
	\rho(A) = \max_{\lambda\in\spec{A}} |\lambda|
\end{align*}
Likewise the radius of the polynomial $\chi_A(\lambda)$,   
\begin{align} \label{eq:poly_spec}
	\rho(\chi_A(\lambda)) = \max_{\lambda\in\spec{\chi_A}} |\lambda|
\end{align}
(Needless to say, the last definition holds for any polynomial).\\

The following matrix decomposition is a simple canonical  representation of any matrix based on its spectrum.\\
A \emph{Jordan block}  $J_k(\lambda)$ is a $k\times k$ matrix of the form,
\begin{align*}
J_k(\lambda) = \mymat{\lambda & 1 \\ & \lambda & 1 \\ && \ddots & \ddots\\ &&& \lambda & 1\\&&&&\lambda  }
\end{align*}
where unspecified entries are zeros.\\
A well-known theorem states that for any $d\times d$ square matrix $A$  there exists an invertible matrix $P$ and $k_1,\dots,k_s$ such that 
\begin{align*}
P^{-1}AP = \oplus_{i=1}^s J_{k_i}(\lambda_i),\qquad \sum_{i=1}^s k_i = d,\qquad \lambda_i\in\spec{A}
\end{align*}
Note that this decomposition is unique up to permutation of the blocks. Therefore, we may define the \emph{index} of an eigenvalue to be the size of its largest Jordan block. As the following theorem demonstrates, the convergence properties of matrices are best analyzed by Jordan form.

\begin{theorem} \label{thm:conv_rate_jord}
Let $A$ be a $d\times d$ square matrix whose spectral radius $\rho_A$ is strictly positive. Then, there exists $C_A>0$  such that for any $k\in\bN$
\begin{align*}
\norm{A^k \bu} \le C_A k^m \rho_A^k\norm{\bu} 
\end{align*}
for any $\bu\in\reals^d$.\\
Furthermore, there exist $c_A >0$ and $\br\in\reals^d$ such that for sufficiently large $k\in\bN$\begin{align*}
\norm{A^k \bu} \ge c_A k^m \rho_A^k\norm{\bu} 
\end{align*}
for any $\bu\in\reals^d$ which satisfy $\inprod{\bu}{\br}\neq0 $.\\
In both cases $m$ denotes the maximal index of eigenvalues of maximum absolute value.
\end{theorem}

\begin{proof}
Let $P$ be a $d\times d $ invertible matrix such that 
\begin{align*}
P^{-1} A P =J
\end{align*}
where $J$ is a Jordan form of $A$, namely, $J$ is a block diagonal matrix such that $J=\oplus_{i=1}^s J_{k_i}(\lambda_i) $ where $\lambda_1,\lambda_2,\dots, \lambda_s$ are eigenvalues of $A$, whose indices 
are $m\eqdef k_1,\dots,k_s$, respectively.  w.l.og we may assume that $\absval{\lambda_1}=\rho_A$ and that $k_1 = \operatorname{index}(\lambda_1)$.
Let  $(Q_1|Q_2|\cdots|Q_s)$ and $(R_1|R_2|\cdots|R_s)$ be the partitions of  the columns of $P$ and the rows of $P^{-1}$, respectively, which conform with the Jordan blocks of $J$. \\ 
Note that for all $i\in[d]$, $J_{k_i}(0)$ is a nilpotent matrix of an order $k_i$. Therefore, for any $(\lambda_i,k_i)$ we have 
\begin{align*}
J_{k_i}(\lambda_i)^k &= (\lambda_i I_{k_i} + J_{k_i}(0) )^k  \\
&= \sum_{j=0}^k \binom{k}{j} \lambda_i^{k-j} J_{k_i}(0)^j\\
&= \sum_{j=0}^{k_i-1} \binom{k}{j} \lambda_i^{k-j} J_{k_i}(0)^j
\end{align*}
Thus,
\begin{align}
J_{k_i}(\lambda_i)^k/ (k^m \lambda_1^k )  
&= \sum_{j=0}^{k_i-1} \frac{\binom{k}{j} \lambda_i^{k-j} J_{k_i}(0)^j}{k^m \lambda_1^k} \nonumber\\
&= \sum_{j=0}^{k_i-1} \frac{\binom{k}{j} } {k^m }
\circpar{\frac{\lambda_i}{\lambda_1}}^k
 \frac{J_{k_i}(0)^j }{\lambda_i^j} \label{eq:jord1}
\end{align}
\\
The proof revolves around the following equality. For any $\bu\in\reals^{pd}$ it holds that 
\begin{align}
\norm{A^k \bu } 
&= \norm{P J^k P^{-1}\bu} \nonumber \\
&= \norm{\sum_{i=1}^s Q_i J_{k_i}(\lambda_i)^k R_i \bu}  \nonumber\\
&= k^m \rho_A^k \norm{\sum_{i=1}^s Q_i\circpar{ J_{k_i}(\lambda_i)/(k^m \lambda_1^k) }R_i \bu}
\end{align}
Plugging in \ref{eq:jord1} yields, 
\begin{align}
\norm{A^k \bu } 
&= k^m \rho_A^k \norm{\underbrace{\sum_{i=1}^s Q_i\circpar{\sum_{j=0}^{k_i-1} \frac{\binom{k}{j} } {k^m }
\circpar{\frac{\lambda_i}{\lambda_1}}^k
 \frac{J_{k_i}(0)^j }{\lambda_i^j} }R_i \bu}_{\bw_k} } \label{eq:jord_main}
\end{align}
Evidently, In order to establish the convergence properties of $A$ it suffices to show that $\{\bw_k\}_{k=1}^\infty$  is a sequence of vectors whose norm can be bounded from above and can be bounded away from zero. \\
\\
Deriving the aforementioned upper bound is straightforward. Indeed,
\begin{align}
\norm{\bw_k} &\le 
\sum_{i=1}^s  \norm{Q_i\circpar{\sum_{j=0}^{k_i-1} \frac{\binom{k}{j} } {k^m }
\circpar{\frac{\lambda_i}{\lambda_1}}^k
 \frac{J_{k_i}(0)^j }{\lambda_i^j} }R_i \bu}\nonumber \\
 &\le \norm{ \bu}
\sum_{i=1}^s  \norm{Q_i}\norm{R_i} \sum_{j=0}^{k_i-1}\norm{ \frac{\binom{k}{j} } {k^m }
\circpar{\frac{\lambda_i}{\lambda_1}}^k
 \frac{J_{k_i}(0)^j }{\lambda_i^j }} \label{eq:tmp2}
\end{align}
Since for all $i\in [d]$ we have
\begin{align*}
\frac{\binom{k}{j} } {k^m }
\circpar{\frac{\lambda_i}{\lambda_1}}^k \to 0 \quad \text{ or } \quad \frac{\binom{k}{j} } {k^m }
\circpar{\frac{\lambda_i}{\lambda_1}}^k =1
\end{align*}
it holds that \ineqref{eq:tmp2} can be bounded from above by some positive scalar $C_A$. Plugging it in into \ref{eq:jord_main} yields
\begin{align*}
\norm{A^k \bu } \le C_A k^m \rho_A^k \norm{\bu}
\end{align*}
\\
Deriving a lower bound on the norm of $\{\bw_k\}$ is a bit more involved. First, we define the following set of Jordan blocks which govern the asymptotic behavior of $\norm{\bw_k}$ 
\begin{align*}
\cI \eqdef  \myset{ i\in[s] }{ \absval{\lambda_i}=\rho_A \text{ and }  k_i = m }
\end{align*}
\eqref{eq:jord1} implies that for all $i\notin\cI$ 
\begin{align*}
J_{k_i}(\lambda_i)^k / (k^m \lambda_1^k ) \to 0  \text{ as } k\to \infty. 
\end{align*}
As for $i\in\cI$, the first $k_i-1$ terms in \eqref{eq:jord1} tend to zero. The last term is a matrix whose entries tend to zero, except for the last entry in the first row which equals
\begin{align*}
\frac{\binom{k}{m-1} } {k^m }
\circpar{\frac{\lambda_i}{\lambda_1}}^k
1/(\lambda_i^{m-1}) 
&\approx
\circpar{\frac{\lambda_i}{\lambda_1}}^k
1/(\lambda_i^{m-1}) 
\end{align*}
By denoting the first column of each $Q_i$ by $q_i$ and the last row in each $R_i$ by $r_i^\top$, we get
\begin{align*}
\norm{\bw_k} &\approx 
\norm{\sum_{i\in\cJ} \circpar{\frac{\lambda_i}{\lambda_1}}^k
\frac{1}{\lambda_i^{m-1}} Q_i J_{m}(0)^{m-1} R_i \bu}\\
&=
\norm{\sum_{i\in\cJ} \circpar{\frac{\lambda_i}{\lambda_1}}^k
\frac{r_i^\top  \bu}{\lambda_i^{m-1}}  q_i }\\
&=  \norm{
\circpar{ \begin{array}{c|c|c|c} &&\\&&\\ q_1&q_2&\cdots&q_{|\cJ|} \\&&\\&&\\\end{array}} 
\mymat{ \frac{r_1^\top  \bu}{\lambda_1^{m-1}}    \\
\circpar{\frac{\lambda_2}{\lambda_1}}^k
\frac{r_2^\top  \bu}{\lambda_2^{m-1}}  \\ \vdots \\ 
\circpar{\frac{\lambda_{|\cJ|}}{\lambda_1}}^k
\frac{r_{|\cJ|}^\top  \bu}{\lambda_{|\cJ|}^{m-1}}  
}}\\
&\ge \min_{\balpha\in K} \norm{\circpar{ q_1|q_2|\cdots|q_{|\cJ|}} \balpha} 
\end{align*}
Where $K$ denotes the compact set $$K\eqdef\myset{\balpha \in\reals^{|\cJ|}}{ \forall i,~|\alpha_i|=\frac{r_i^\top  \bu}{\lambda_i^{m-1}}  }$$
Now, if $\bu$ satisfies $r_1^\top\bu\neq0$ then this implies that $0\notin K$. Furthermore, since $\circpar{ q_1|q_2|\cdots|q_{|\cJ|}}$ is of full rank it holds for the following continuous function that
\begin{align*}
\circpar{ q_1|q_2|\cdots|q_{|\cJ|}} \balpha\neq 0,\quad\forall\balpha\in K.
\end{align*}
The extrema points of continuous function over compact set are always attainable, therefore 
\begin{align}
\min_{\balpha\in K} \norm{\circpar{ q_1|q_2|\cdots|q_{|\cJ|}} \balpha} >0
\end{align} 
By which we conclude that there exists $c_A>0$ such that $\norm{\bw_k}>c_A$ for sufficiently large $k$. Plugging it in into \eqref{eq:jord_main}  yields
\begin{align*}
\norm{A\bu} \ge c_A k^m \rho_A^k\norm{\bu} 
\end{align*}
for any $\bu\in\reals^d$ such that $\inprod{\bu}{\br_1}\neq0$ and for sufficiently large $k$.
\end{proof}

An immediate corollary of \thmref{thm:conv_rate_jord} is that if a $p$-CLI method converges then its asymptotic convergence rate must be linear, i.e. the number of the correct bits of the current approximation grows linearly - unless the spectral radius of its iteration matrix is 0. In which case the asymptotic convergence rate is  \emph{quadratic}, meaning that the number of correct bits grows quadratically. The latter case is out of the scope of this work.

A matrix $A$ is \emph{Convergent} if  $\lim_{k\to\infty}A^k=0$. 
The following is a direct corollary of \thmref{thm:conv_rate_jord}.
\begin{corollary} \label{cor:conv_equi}
Let $A$ be a square matrix. Then the following statements are equivalent:
\begin{itemize}
\item $\rho(A)<1$.
\item $\lim_{k\to\infty}A^k=0$.
\item The Neuman series $\sum_{k=0}^\infty A^k$ converges.
\end{itemize}
In which case, $(I-A)^{-1}$ exists and $(I-A)^{-1}=\sum_{k=0}^\infty A^k$.
\end{corollary}
\begin{proof}
The first two statements are readily implied by \thmref{thm:conv_rate_jord}. The rest can be proved using the following identity
\begin{align}
(I-A)\sum_{k=0}^{m-1} A^k = I-A^{m}
\end{align}
which holds for any $m\in\bN$.
\end{proof}
Note that any norm that the real vector space $\reals^{d\times d}$ might be endowed with induces the same topology. Hence, we may freely use notions of convergence of sequences and series of matrices without specifying which norm is being used.\\

Establishing tools which allow close examination of the convergence properties of general square matrices, we may now derive the following characterization for convergence of $p$-CLI methods.

\begin{theorem} \label{thm:conv_cli}
Let $\cM$ be a CLI iterative method whose iteration matrix and free summand, $M$ and $\bv$, are drawn randomly according to some distribution $\cD_{pd\times pd}$ over $pd\times pd$ matrices and some distribution $\cD_{pd\times 1}$ over $pd$-dimensional vectors, respectively.\\
Then $\circpar{\bE \bz^k}_{k=0}^\infty$  converges for any $\bz^0\in\reals^d $ if and only if $\rho(\bE_{D_{pd}}[ M])<1$.\\
In which case, we have that for any $\bz^0\in\reals^d,\quad \bE \bz^k \to \bz^*\eqdef \circpar{I-\bE M}^{-1}\bE \bv $. That is, the limit of this process is the fix point of the expected update rule.
\end{theorem}
\begin{proof} 
Fix $k\in\bN$ and denote by $M^{(0)},\dots,M^{(k-1)}$ the first $k$ randomly drawn iteration matrices
and by $\bv^{(0)},\dots,\bv^{(k-1)}$ the first $k$ randomly drawn free summands.
By \eqref{Eq:canonical_dynamic} we get 
\begin{align*}
\bz^1 &= M^{(0)} \bz^0 + \bv^{0}\\
\bz^2 &= M^{(1)} \bz^1 + \bv^{1} \\
&=M^{(1)} M^{(0)} \bz^0 + M^{(1)}\bv^{0}+ \bv^{1}\\
\bz^3 &= M^{(2)} \bz^2 + \bv^{2} \\
&=M^{(2)}M^{(1)} M^{(0)} \bz^0 + M^{(2)}M^{(1)}\bv^{0}+ M^{(2)}\bv^{1} +\bv^{2}\\
\vdots\\
\bz^k &=  \prod_{j=0}^{k-1} M^{(j)}\bz^0 + \sum_{m=1}^{k-1} \prod_{j=m}^{k-1} M^{(j)}\bv^{m-1} + \bv^{k-1}\\
&=  \prod_{j=0}^{k-1} M^{(j)}\bz^0 + \sum_{m=1}^{k} \circpar{ \prod_{j=m}^{k-1} M^{(j)} }\bv^{m-1} 
\end{align*}
Where, by convention, we set an empty product to be the identity matrix and define the order of multiplication of factors of abbreviated product notation to be carried out from the highest index to the lowest, i.e. $\prod_{j=1}^k M^{(j)} = M^{(k)}\cdots  M^{(1)} $. \\
Now, by the linearity of the expectation operator and by the independency structure of the iteration matrix and the inversion matrix we get a more explicit expression for $\bE \bz^k$,
\begin{align}
\bE \bz^k &= \bE [M]^k \bz^0+ \circpar{\sum_{j=0}^{k-1} \bE [M]^j}  \bE \bv \label{EqLine:e25}
\end{align}
Where the expectation is taken over $(\cD_{pd\times pd })^k$ and $(\cD_{pd\times 1})^k$, accordingly. For the sake of brevity, we will omit these subscripts from the expectation operator in the reminder of the work.\\
\\
By \corref{cor:conv_equi} we have that $\sum_{l=0}^{\infty} \bE [M]^l$ converges if and only if $\rho(\bE [ M])<1$. Thus, clearly if $\rho(\bE [ M])<1$ then $\circpar{\bE \bz^k}_{k=0}^\infty$ converges for any $\bz^0\in\reals^d$.\\
\\
Now, assume that $\circpar{\bE \bz^k}_{k=0}^\infty$ converges for any $\bz^0\in\reals^d$. The second summand in \eqref{EqLine:e25} must converge, for otherwise, by taking $\bz^0 =0$ we get that $\circpar{\bE \bz^k}_{k=0}^\infty$ diverges.
This, in turn, implies that for any $\bz^0$, the sequence $\bE[M]^k \bz^0$, being a difference of two convergent series, also converges. By which we conclude $\rho(\bE [M])<1$.\\
\\
In case of convergence, we may use \corref{cor:conv_equi} again together with \eqref{EqLine:e25} to obtain,
\begin{align*}
\bE \bz^k = \bE [M]^k \bz^0+ \circpar{\sum_{l=0}^{k-1} \bE [M]^i}  \bE \bv \to \circpar{I-\bE M}^{-1}\bE \bv 
\end{align*}
\end{proof}

We now turn to a more delicate analysis of the convergence properties of $p$-CLI methods. First, we define the \emph{Spectrum} of a $p$-CLI iterative method $\cM$ 
$$ \spec{\cM}=\spec{\bE M}$$
and its \emph{Spectral Radius}
$$\rho_{\cM} \eqdef \rho(\bE M)$$
Naturally, A $p$-CLI whose spectral radius is strictly smaller than one is said to be \emph{Convergent}. \\
After deriving a condition for convergence of $p$-CLI method, an equally important question arises: What is the convergence rate of a given convergent $p$-CLI method? Put differently, how fast does $\norm{\bE \left[\bz^k\right]-\bz^*}$ vanish
\footnote{A word of caution - whereas we use an analytically more convenient interpretation of the error magnitude $\norm{\bE \left[\bz^k\right]-\bz^*}$, the 'slightly' different measure $\bE \left[\norm{\bz^k-\bz^*}\right]$ is unquestionably more interesting.}
?\\
The next theorem gives a complete answer for this question.

\begin{theorem} \label{thm:conv_rate}
Let $\cM$ be a convergent $p$-CLI iterative method whose spectral radius $\rho_M$ does not vanish, and let us denote its expected limit point by $\bz^*$. \\
Then, there exist $C_\cM>0$ such that for all $\bz^0 \in \reals^d$ it holds that 
\begin{align*}
\forall k\in\bN,~\norm{\bE \left[\bz^k\right]-\bz^*} \le C_\cM k^m \rho_M^k\norm{\bz^0-\bz^*} 
\end{align*}
Furthermore, there exist $c_\cM>0$ and $\br\in\reals^d$ such that for  any $\bz^0\in\reals^d$ which satisfy $\inprod{\bE [\bz^0]-\bz^*}{\br}\neq0 $, and for sufficiently large $k\in\bN$ it holds that
\begin{align*} 
\norm{\bE \left[\bz^k\right]-\bz^*} \ge c_\cM k^m \rho_A^k\norm{\bz^0-\bz^*} 
\end{align*}
In both cases $m$ denotes the maximal index of eigenvalues of maximum absolute value.
\end{theorem}

\begin{proof}
Recall that by \eqref{EqLine:e25},
\begin{align*}
\bE \bz^k &= \bE [M]^k \bz^0+ \circpar{\sum_{l=0}^{k-1} \bE [M]^i}  \bE \bv \\
&= \bE [M]^k \bz^0+ \circpar{I-\bE M}^{-1}\circpar{I-\bE M}\circpar{\sum_{l=0}^{k-1} \bE [M]^i}  \bE \bv \\
&= \bE [M]^k \bz^0+ \circpar{I-\bE M}^{-1}\circpar{I-\bE[M]^k }  \bE \bv 
\end{align*} 
Moreover, by \thmref{thm:conv_cli} we have
\begin{align*}
\bz^*= \circpar{I-\bE M}^{-1}\bE \bv
\end{align*}
Thus, 
\begin{align} \label{eq:conv_eq2}
\norm{\bE \bz^k - \bz^*} 
&= \norm{\bE [M]^k \bz^0+ \circpar{I-\bE M}^{-1}\circpar{I-\bE[M]^k }  \bE \bv  - \circpar{I-\bE M}^{-1}\bE \bv} \nonumber\\
&= \norm{\bE [M]^k \bz^0- \circpar{I-\bE M}^{-1}\bE[M]^k   \bE \bv  } \nonumber\\
&= \norm{\bE [M]^k \circpar{ \bz^0- \circpar{I-\bE M}^{-1}\bE \bv } } \nonumber\\
&= \norm{\bE [M]^k \circpar{ \bz^0- \bz^* } } 
\end{align} \\
Applying theorem \thmref{thm:conv_rate_jord} on $\bE[M]$ concludes the proof
\end{proof}

\begin{remark}
One may access a much wider family of iterative methods by omitting the assumption regarding the stationarity of the distributions of the iterations matrices. This relaxation leads to a measure called the \emph{Joint Spectral Radius} defined for a set of matrices $M=\left\{A_{i_1},\dots,A_{i_k}\right\}\subseteq \reals^{d\times d}$ by,
\begin{align}
\rho(M) =\lim_{k\to\infty} \max \left\{  \norm{ A_{i_1}\cdots A_{i_k}}^{1/k} : A_i\in M \right\}
\end{align}
In spite of the generality of this approach,  there are negative theoretical results which show that the joint spectral radius is very hard to compute or to approximate.  
\end{remark}

\begin{remark}
Similar results may be obtained for differentiable nonlinear iterative methods by applying the same techniques used in the last section using gradients (e.g. \cite{polyak1987introduction}).  
\end{remark}

\section{Spectrum} \label{section_cano}
By now it should be evident why the spectral radius of $p$-CLI methods is crucial in determining their convergence rates. The next section is exclusively devoted for analyzing the spectral radius of $p$-CLI methods by inspecting characteristic polynomials of iteration matrices.\\

Let $\cM$ be a $p$-CLI method over $\reals^d$ and denote its expected iteration matrix by $M$. When there is no of risk of disambiguation we omit the expectation operator symbol, e.g. in this section $M$ replaces $\bE M$. According to \thmref{thm:conv_rate}, in order to bound its convergence rate, we need to bound its spectral radius $\rhoM$. To this end, recall that
\begin{align*}
M = \mymat{0_d & I_d &&& \\&0_d & I_d&&\\ &&&&\\ && \ddots&\ddots\\&&& \\&&& 0_d & I_d\\C_{0}&&\dots&C_{p-2}&C_{p-1} }
\end{align*} 
for some $C_0,\dots,C_{p-1}\in\reals^d$. \\

The following lemma provides an explicit expression for $\chiM(\lambda)$, the characteristic polynomial of $M$. It is worth pointing out that this lemma is proven in a straightforward manner, although there are other more compact, but perhaps less direct, proofs.

\begin{lemma} \label{lemma:mat_pol_spec}
Using the notation above, we have, 
\begin{align}
\chiM(\lambda) =  (-1)^{pd} \det\circpar{\lambda^p I_d - \sum_{k=0}^{p-1} \lambda^{k} C_k}
\end{align}
\end{lemma}
\begin{proof}
For $\lambda\neq0$ we get, 
\begin{align*}
\chiM(\lambda) &= \det(M-\lambda I_{pd}) \\
&= \det\circpar{\begin{array}{cccc|c} -\lambda I_d & I_d &&& \\&-\lambda I_d & I_d&&\\ &&&&\\ && \ddots&\ddots\\&&& \\&&& -\lambda I_d & I_d\\ \hline C_{0}&&\dots&C_{p-2}&C_{p-1}-\lambda I_d \end{array}}
\end{align*}
\begin{align*}
&= 
\det\circpar{\begin{array}{cccc|c} -\lambda I_d & I_d &&& \\&-\lambda I_d & I_d&&\\ &&&&\\ && \ddots&\ddots\\&&& \\&&& -\lambda I_d & I_d\\ \hline 0_d&C_{1}+\lambda^{-1}C_0&\dots&C_{p-2}&C_{p-1}-\lambda I_d \end{array}}\\
&= \det\circpar{\begin{array}{cccc|c} -\lambda I_d & I_d &&& \\&-\lambda I_d & I_d&&\\ &&&&\\ && \ddots&\ddots\\&&& \\&&& -\lambda I_d & I_d\\ \hline 0_d&0_d&
C_2 + \lambda^{-1}C_{1}+\lambda^{-2}C_0\dots &C_{p-2}&C_{p-1}-\lambda I_d \end{array}}\\
&= \det\circpar{\begin{array}{cccc|c} -\lambda I_d & I_d &&& \\&-\lambda I_d & I_d&&\\ &&&&\\ && \ddots&\ddots\\&&& \\&&& -\lambda I_d & I_d\\ \hline 0_d&\dots &0_d&&
\sum_{k=1}^p \lambda^{k-p} C_{k-1} -\lambda I_d \end{array}}\\
&= \det(-\lambda I_d)^{p-1} \det\circpar{\sum_{k=1}^p \lambda^{k-p} C_{k-1} -\lambda I_d }
\end{align*}
\begin{align*}
&= (-1)^{(p-1)d} \det\circpar{\sum_{k=1}^p \lambda^{k-1} C_{k-1} -\lambda^p I_d}\\
&= (-1)^{pd} \det\circpar{\lambda^p I_d - \sum_{k=0}^{p-1} \lambda^{k} C_k}
\end{align*}
For $\lambda=0$, it is clear that $\chiM(0)=0$ if and only if $\det(C_0)=0$. Thus, the equality holds for all $\lambda$.
\end{proof}

Let $L(\lambda)$ denote the following \emph{Matrix Polynomial}, i.e. a polynomial whose coefficients are matrices, 
\begin{align} \label{eq:char_matrix_polynomial}
L(\lambda)= \lambda^p I_d - \sum_{k=0}^{p-1} \lambda^{k} C_k
\end{align}
\eqref{eq:char_matrix_polynomial} is the reason why $C_i$ are called the coefficient matrices. It is conventional to call $M$ the \emph{Companion matrix} of $L(\lambda)$. For more details on this topic see \cite{gohberg2009matrix}.\\
\\
In order to derive a more explicit expression of $\chiM(\lambda)$, we may inspect a narrower family of coefficient matrices $C_i$, in which all $C_i$ are simultaneously triangularizable. As we will shortly see, this additional assumption on $C_i$ holds true for most of the $p$-CLI methods of interest. For instance, it is common to have all $C_i$ as polynomial expressions in some matrix, in which case the hypothesis is satisfied.

\begin{theorem} \label{thm:char_poly}
Using the notation above, assume that $C_0,C_1,\dots,C_{p-1}$ form a collection of simultaneously triangularizable matrices. That is, there exists an invertible matrix $Q\in\reals^{d\times d}$ such that 
\begin{align*}
	 T_i\eqdef  Q^{-1} C_i Q \quad i=0,1,\dots,p-1
\end{align*}
are upper triangular matrices. Then
\begin{align*}
\chiM(\lambda) =(-1)^{pd} \prod_{j=1}^d \circpar{\lambda^p - \sum_{k=0}^{p-1} \lambda^{k} (T_{k})_{j,j}}
\end{align*}
\end{theorem}

\begin{proof}
By hypothesis there exists an invertible matrix $Q\in\reals^{d\times d}$ such that 
\begin{align*}
	 T_i\eqdef  Q^{-1} C_i Q \quad i=0,1,\dots,p-1
\end{align*}
Using \ref{lemma:mat_pol_spec} we see that, 
\begin{align*}
\chiM(\lambda)
&= (-1)^{pd} \det\circpar{\lambda^p I_d - \sum_{k=0}^{p-1} \lambda^{k} C_{k}}\\
&= (-1)^{pd} \det\circpar{Q(\lambda^p I_d) Q^{-1} - \sum_{k=0}^{p-1} \lambda^{k} QT_{k}Q^{-1}}\\
&= (-1)^{pd} \det(Q) \det\circpar{\lambda^p I_d - \sum_{k=0}^{p-1} \lambda^{k} T_{k}} \det(Q^{-1})\\
&= (-1)^{pd} \det\circpar{\lambda^p I_d - \sum_{k=0}^{p-1} \lambda^{k} T_{k}} &\nonumber\\
&= (-1)^{pd} \prod_{j=1}^d\circpar{\lambda^p I_d - \sum_{k=0}^{p-1} \lambda^{k} T_{k}}_{j,j}  \\
&= (-1)^{pd} \prod_{j=1}^d{\circpar{ \lambda^p - \sum_{k=0}^{p-1} \lambda^{k} (T_{k})_{j,j}  }} \\
\end{align*}
\end{proof}
Note that the elements on the diagonal of $T_i$ in the last theorem, produced by some invertible matrix $Q$, are exactly the eigenvalues of $C_i$ in an arbitrary order. Let $\text{Eig}_j^Q(C_i)$ denote the $j$'th eigenvalue of $C_i$ according to the order induced by $Q$. The following is a direct consequences of this remark and \thmref{thm:char_poly}.
\begin{corollary}  \label{coro:cli_spec}
Let $\cM$ be a $p$-CLI method over $\reals^d$ with coefficient matrices $C_0,C_1,\dots,C_{p-1}$ simultaneously triangularized by some invertible matrix $Q$. Then,
\begin{align}   \label{eq:cli_spec}
\spec{\cM} =\bigcup_{j=1}^d \spec{\lambda^p - \sum_{k=0}^{p-1} \lambda^{k} \sigma^Q_j(C_k)}
\end{align}
\end{corollary}
As we will see, \eqref{eq:cli_spec} is of great importance for deriving lower bounds, as well as for designing optimization algorithms.
\clearpage

\chapter{p-CLI Optimization Algorithms}  \label{chapter:pcli_algos} 
Being well-understood, one may benefit a lot from describing the dynamics of various processes as $p$-CLI methods. In what follows we shall form a bridge between a certain class of optimization algorithms, which we call $p$-CLI optimization algorithms, and the analytic theory of polynomials by using this idea. \\

First, motivated by the example shown in Section \ref{section:sdca_case_study}, we establish the framework under which we will be analyzing optimization algorithms. Next, we formulate various popular optimization algorithms in this framework so as to demonstrate its utility. We then discuss the convergence properties of such algorithms as well as their iteration complexity. Lastly, we derive a fundamental principle which relates the sum of the coefficient matrices of any $p$-CLI optimization algorithm to its inversion matrix. 

\section{Framework} \label{section:framework}
We saw in Section \ref{section:sdca_case_study} that by applying SDCA on quadratic functions we were able to express it in terms of some iterative method $\cM$. This, in turn, enabled us to closely inspect the dynamics of the process. Indeed, by computing the iteration matrix of $\cM$ along with its spectrum we were able to derive an estimation for any test point of $\cM$ generated after finite number of iterations. In the sequel we introduce a framework which generalizes the arguments used for SDCA.\\

For the sake of brevity, we identify each quadratic function of the form 
\begin{align} 
f(\bx) = \frac{1}{2}\bx^\top A\bx + \bb^\top \bx
\end{align}
with the ordered pair $(A,\bb)$. \\
We further define,
\begin{align*}
\posdefun{d}{\Sigma} = \{(A,\bb)|~ A \in\reals^{d\times d} \text{  is symmetric and }\spec{A} \subseteq \Sigma,~\bb\in\reals^d\}
\end{align*}
for some $\Sigma\subseteq\reals^{++}$. Typically we have  $\Sigma = [\mu,L]$ for some prescribed constants $0<\mu<L$, in which case $(A,\bb)$ represents an $L$-smooth $\mu$-strongly convex function,  as implied by \thmref{ineq_conv_hessian}.\\
Now, let $\cA$ be an optimization algorithm which minimizes smooth strongly convex functions. We may apply $\cA$ on any quadratic function in $\posdefun{d}{\Sigma}$. If the resulting dynamics of $\cA$ can be cast as some iterative method, then one might hope to use the tools derived in the preceding chapter. Indeed, the formulation we choose, casting optimization algorithms as $p$-CLI methods, is just one of many possible ways for generalizing the SDCA case.  That being said, the advantages of this formulation are applicability to a wide range of popular algorithms and the relative simple analytic nature of $p$-CLI methods.\\

Formally in this framework we examine any iterative optimization algorithm $\cA$ 
such that when applied on $\posdefun{d}{\Sigma}$ takes the form of some iterative method $\cM$, as defined in \ref{def:iterative_method}, i.e. its initialization and update rule are 
\begin{align} \label{iteratio_d}
&\bz^0 = U \bx^0\\
&\bz^{k+1} = M(A,k) \bz^k + U N(A,k)\bb  \label{iterstep}
\end{align}
for some \emph{Lifting Factor} $p\in\bN$, a \emph{Lifting Matrix} $U\in \reals^{pd\times d}$ with  $U^\top U=I_d$, randomly drawn $pd\times pd$ \emph{Iteration Matrices} $M(A,k)$ and randomly drawn $d\times d$ \emph{Inversion Matrices}  $N(A,k)$.  The reason why we choose to call $N(A,k)$ an inversion matrix, will become clear in  Chapter \ref{chapter:lower_bounds}. Furthermore, the term iteration matrix may refer to the mappings $M(X,k):\posdefun{d}{\Sigma}\to \cD_{\reals^{pd\times pd}}$ or to a specific evaluation of it at $A$. It will be clear from the context which interpretation is used. The same convention holds for inversion matrices.\\

Any execution  of the algorithm induces a sequence of test points on $\reals^d$ by
\begin{align*}
\bx^k\eqdef U^\top \bz^k
\end{align*}
Assuming that $\cA$ minimizes the quadratic function at hand implies that this sequence gets closer and closer to the minimizer $\bx^*\eqdef -A^{-1}\bb$, regardless of the initialization point. Note that this yields,
\begin{align} \label{eq:conv_sol}
\bx^* = U^\top \bz^* 
\end{align}
\\

As one might expect, we are particularly  interested in algorithms which when applied on quadratic functions may take the form of a $p$-CLI method as defined in \ref{def:pcli}. We call such algorithms \emph{$p$-CLI Optimization Algorithms}. In particular, being stationary the iteration matrices and inversion matrices as defined in \ref{iteratio_d} must satisfy 
\begin{align*}
\forall k, ~M(A,k)=M(A,0)\\
\forall k, ~N(A,k)=N(A,0)
\end{align*}
(Here equality indicates an identical matrix distribution for all $k$)\\
in which case we denote $M(A)\eqdef M(A,0)$ and $N(A)\eqdef N(A,0)$. \\ We further require that the lifting matrix  of $p$-CLI methods would satisfy 
\begin{align} \label{def:U_def}
U = E_{p} \eqdef  (\underbrace{0_d,\dots,0_d}_{p-1 \text{ times}}, I_d)^\top
\end{align}

Finally, throughout the rest of this monograph, we will add the following technical requirement: The coefficient matrices of any $p$-CLI optimization algorithms are assumed to be simultaneously triangularizable. This assumption is most likely to be removed by using tools from the theory of Matrix Polynomials. We defer such treatment to future work. That being said, in all cases that we have examined this technical assumption is satisfied. Informally, the reason is that $C_i(A)$ cannot have a complex dependence on $A$, for otherwise one might spend more time on computing $C_i(A)$ than on the optimization process itself. Indeed, it is common to choose $C_i(A)$ as linear polynomials in $A$ or diagonal matrices with simple dependence on $A$, in which case the assumption holds true.\\

As we will see shortly, these algorithms lend themselves well to the theory of CLI methods, thus allowing one to readily apply the results derived in \ref{section:cli_conv_prop}.

\section{Specifications for Popular algorithms} \label{section_spec_algo}
Naturally, one might ask how wide-ranging is this framework, and what does characterize 
the optimization algorithms which it applies to. Roughly speaking, any algorithm whose rule of update is linearly dependent on the gradient and the hessian of the function under consideration may fit this technique. Instead of providing a precise answer for this question, we give specifications of various popular optimization algorithms for some quadratic function $f(\bx) = \frac{1}{2}\bx^\top A \bx + \bb^\top \bx$ in $\posdefun{d}{[\mu,L]}$.\\

\begin{itemize}
\item Full Gradient Descent (FGD) (see \cite{nesterov2004introductory})- \label{spec:fgd}
			\begin{align*}
			\bx^0 &\in \reals^d\\
			\bx^{k+1} &= \bx^k - \beta \nabla f(\bx^k)= \bx^k - \beta(A\bx^k +\bb)=(I-\beta A)\bx^k -\beta \bb\\
			\beta &= \frac{2}{\mu +L}
			\end{align*}
			Thus, this version of FGD forms a $1$-CLI optimization algorithm  such that,
			\begin{align*}
			M(A)&= I-\beta A\\
			N(A) &=-\beta I_d
			\end{align*}

\item Newton method (see \cite{nesterov2004introductory}) - \label{spec:newton}
\begin{align*}
\bx^0 &\in \reals^d\\
\bx^{k+1} &=  \bx^k- (\nabla^{2} f(\bx^k))^{-1} \nabla f(\bx^k) = \bx^k- A^{-1}(A\bx^k + \bb)\\&= (I-A^{-1}A)\bx^k - A^{-1}\bb = -A^{-1}\bb
\end{align*}
Hence, Newton method is a $1$-CLI optimization algorithm with
\begin{align*}
M(A)&= 0\\
N(A) &=- A^{-1}
\end{align*}

\item The Heavy Ball Method (\cite{polyak1987introduction})- 
\begin{align*}
\bx^0 &\in \reals^d\\
\bx^{k+1} &= \bx^k - \alpha(A\bx^{k} + \bb)+ \beta (\bx^{k}-\bx^{k-1}) \\&= (I-\alpha A + \beta I) \bx^{k} -\beta I \bx^{k-1} -\alpha \bb  \\&= \circpar{(1+\beta) I-\alpha A } \bx^{k} -\beta I \bx^{k-1} -\alpha \bb\\
\alpha &= \frac{4}{\circpar{\sqrt{L}+\sqrt{\mu}}^2}\\
\beta &= \circpar{\frac{\sqrt{L}-\sqrt{\mu}}{\sqrt{L}+\sqrt{\mu}} }^2
\end{align*}
That is, the Heavy Ball method is a $2$-CLI optimization algorithm with
\begin{align*}
M(A)&= \circpar{\begin{array}{cc} 0 & I \\  -\beta I &(1+\beta) I-\alpha A  \end{array}}\\
N(A) &=- \alpha I_d
\end{align*}

\item Accelerated Gradient Descent (AGD) (see \cite{nesterov2004introductory})-\label{spec:agd}
\begin{align*}
\bx^0&=\by^0 \in \reals^d\\
\by^{k+1} &= \bx^{k} - \frac1L\nabla f(\bx^{k})\\
\bx^{k+1} &= \circpar{1+\alpha}\by^{k+1} - \alpha \by^{k} \\
\alpha &= \frac{\sqrt{L}-\sqrt{\mu}}{\sqrt{L}+\sqrt{\mu}}
\end{align*}
Which can be rewritten as,
\begin{align*}
\bx^0& \in \reals^d\\
\bx^{k+1} &= \circpar{1+\alpha}\circpar{\bx^{k} - \frac1L\nabla f(\bx^{k})} - \alpha \circpar{\bx^{k-1} - \frac1L\nabla f(\bx^{k-1})}\\
&= \circpar{1+\alpha}\circpar{\bx^{k} - \frac1L (A\bx^{k}+\bb)} - \alpha \circpar{\bx^{k-1} - \frac1L (A\bx^{k-1}+\bb)}\\
&= \circpar{1+\alpha}\circpar{I - \frac1L A} \bx^{k}
 -\alpha\circpar{I - \frac1L A} \bx^{k-1}
 -\frac1L \bb
\end{align*}
AGD is $2$-CLI optimization algorithm with 
\begin{align*}
M(A)&= \circpar{\begin{array}{cc} 0 & I \\ -\alpha\circpar{I - \frac1L A}  & \circpar{1+\alpha}\circpar{I - \frac1L A} \end{array}}\\
N(A) &=- \frac1L I_d \\
\end{align*}

\item Stochastic Gradient Descent (SGD) (e.g. \cite{kushner2003stochastic,spall2005introduction})- A straightforward extension of the deterministic version FGD goes as follows.\\
 Let $(\Omega,\cF,\cP)$ be a probability space  and let $G(\bx,\omega):\reals^d\times \Omega\to\reals^d$ be 
an unbiased estimator of $\nabla f(\bx)$ for any $\bx$. That is,
\begin{align*}
\bE[ G(\bx,\omega) ] &= \nabla f(x) = A \bx + \bb
\end{align*}
As in (\cite{nemirovski2005efficient}), SGD may be specified by
\begin{align*}
			&\bx^0 \in \reals^d\\
			&\text{Generate } \omega_{k} \text{ randomly and set } \bx^{k+1} = \bx^k - \gamma_i G(\bx^k,\omega_{k})\\
			&\gamma_i = \frac{1}{\mu i}
			\end{align*}
			It is common to assume that $$G(\bx,\omega)= \nabla f(\bx) +\be(\bx,\omega)$$ and 
			\begin{align} \label{eq:zero_mean_noise}
			\bE[\be(\bx,\omega)]=0 \quad ,\forall \bx\in\reals^d
			\end{align}
			in which case we have
			\begin{align*}
			\bx^{k+1} &= \bx^k - \gamma_i \circpar{ \nabla f(\bx^k) +\be(\bx,\omega) }\\
			&= \bx^k - \gamma_i \circpar{ A\bx^k +\bb +\be(\bx,\omega) }\\
			&= \circpar{I-\gamma_i A}\bx^k  - \gamma_i \bb - \gamma_i \be(\bx,\omega)
			\end{align*}
			Evidently, there are types of noise for which the resulting algorithm may not form a $p$-CLI optimization algorithm. Nevertheless, if $$\be(\bx,\omega)=A_\omega\bx + \bb_\omega$$ for some $A_\omega$ and $\bb_\omega$ which satisfy
\begin{align*}
\bE[A_\omega]&=0\\
\bE[b_\omega]&=0
\end{align*}
then for appropriately chosen $\gamma_i$, e.g. $\gamma_i=\frac{2}{\mu+L}$, we get a stochastic iterative method which forms a $p$-CLI optimization algorithm. 
			
\item Stochastic Coordinate Descent (SCD) (see \cite{shalev2013stochastic})- This is an  extension of the example shown in Section \ref{section:sdca_case_study}. SCD works by repeatedly minimizing a uniformly randomly drawn coordinate of $f$ in each iteration. Which may be expressed as
\begin{align*}
&\bx^0 \in \reals^d\\
&\text{Pick } i\sim \cU([d]) \text{ and set }\bx^{k+1} =  \left(I-\frac{1}{A_{i,i}}\be_i a_i^\top \right)\bx^k - \frac{b_i}{A_{i,i}}\be_i
\end{align*}
where $a_i^\top$ denotes the $i$'th row of $A$ and $\bb\eqdef\circpar{b_1,b_2,\dots,b_d}$.  \\
That is, SCD is a $1$-CLI optimization algorithm such that 
\begin{align*}
\bE M(A)&= I-\frac{1}{d}\text{Diag}^{-1}(A_{1,1},\dots,A_{d,d})A\\
\bE N(A) &=- \frac{1}{d}\text{Diag}^{-1}(A_{1,1},\dots,A_{d,d})
\end{align*}
As a matter of fact, this method is equivalent to the well-known Jacobi's iterative method.

\item Conjugate Gradient Descent (CGD) (see \cite{nemirovski2005efficient})  can be re-expressed as a non-stationary linear iterative method. 

\item Stochastic Average Gradient (SAG) (see \cite{roux2012stochastic}) - Much like SDCA, SAG is an optimization algorithms that is designed to solve an optimization problem which is closely related to RLM (defined in \ref{opt:RLM}). SAG forms an non-trivial example of casting an optimization algorithm as a stationary linear iterative method. For the exact derivation see Section \ref{section:SAG}. Unfortunately, this derivation results in an iteration matrix which is slightly different from those of $p$-CLI optimization algorithms. Thus, in order to analyze SAG in this framework a further generalization is required.
\end{itemize}

\section{Convergence Properties}

A straight forward implication of \thmref{thm:conv_cli} is that a $p$-CLI optimization algorithm over $\reals^d$ converges if and only if the spectral radius of its iteration matrix $M$ is strictly smaller than 1. In which case, the algorithm converges to the fixed point 
 $$\circpar{I-\bE M(A)}^{-1}\bE \left[ U N(A)\bb\right] $$
\thmref{thm:conv_rate} comes in handy when determining the convergence rate. The upper bound shown in this theorem implies that there exists $C_\cM>0$ and $m\in\bN$ such that for any initial point it holds that 
\begin{align*}
\forall k\in\bN,~\norm{\bE \left[\bz^k\right]-\bz^*} \le C_\cM k^m \rho(M)^k\norm{\bz^0-\bz^*} 
\end{align*} 
Therefore for all $k\in\bN$ we have,
\begin{align}
\norm{\bE \left[\bx^k\right]-\bx^*} &=
\norm{\bE \left[U^\top\bz^k\right]- U^\top\bz^*} \nonumber \\
&\le \norm{U^\top} \norm{\bE \left[\bz^k\right]-\bz^*} \nonumber\\
&\le \norm{U^\top}C_\cM k^m \rho(M)^k\norm{\bz^0-\bz^*} \nonumber\\
&= \norm{U^\top}C_\cM k^m \rho(M)^k\norm{U\bx^0-U\bx^*} \nonumber\\
&\le \norm{U^\top}\norm{U}C_\cM k^m \rho(M)^k\norm{\bx^0-\bx^*} \label{ineq:x_up}
\end{align}
The lower bound, on the other hand,  states that  there exist $c_\cM>0,~m\in\bN$ and $\br\in\reals^{pd}$ such that for  any $\bz^0\in\reals^d$ which satisfy $\inprod{\bz^0-\bz^*}{\br}\neq0 $, and for sufficiently large $k\in\bN$ it holds that
\begin{align} \label{ineq:conv_lb}
\norm{\bE \left[\bz^k\right]-\bz^*} \ge c_\cM k^m \rho(M)^k\norm{\bz^0-\bz^*} 
\end{align}
Thus, in order to properly apply this lower bound we must first show that there exists $\bx^0\in\reals^d$ such that $U\bx^0-U\bx^*$ has a non-trivial projection on $r$. In what follows, we will prove a slightly weaker theorem, namely, we show that there exists $\bx_0\in\reals^d$ such that after at most $p$ iteration $\bE [\bx^k]-\bx^0$ has a non-trivial projection onto $\br$. Since \ref{ineq:conv_lb} holds for sufficiently large $k$, the lower bound we obtain for $\bE[\bx^k] -\bx^*$ is essentially the same.\\

Figuratively speaking, an important property of the expected iteration matrix $\bE M$ of $p$-CLI optimization algorithms is that it allows one to 'exhaust' the whole space $\reals^{pd}$ by recurrent matrix multiplication. Iterating a $p$-CLI process induces 'many' paths in $\reals^{pd}$, instead of $\reals^d$, which results in an acceleration of the convergence 
rate. The following theorem put this intuition on a rigorous ground.
\begin{theorem} \label{thm:cli_ex}
Let $\bv\in\reals^{pd}$ . Then, there exist $\bu\in\reals^d$ and $k\in\{0,\dots,p-1\}$ such that $\inprod{(\bE M)^k E_p\bu}{\bv}\neq 0 $.
\end{theorem}
\begin{proof}
Let $B=\{\be_1,\dots,\be_d\}$ be any basis for $\reals^d$. We claim that the following form a basis for $\reals^{pd}$
\begin{align*}
\cB =\{&E_p\be_1, ME_p\be_1, M^2 E_p\be_1,\dots ,M^{p-1}E_p\be_1,\\
&E_p\be_2, M E_p\be_2, M^2 E_p\be_2,\dots ,M^{p-1} E_p\be_2\\
& \qquad\qquad\qquad\qquad\vdots\\
&E_p\be_d, M E_p \be_d, M^2 E_p\be_d,\dots,  M^{p-1} E_p\be_d \}
\end{align*}
Proving this concludes the proof, for if
\begin{align*}
< M^k E_p\be_i,\bv>=0, \qquad  \forall k=0,\dots,p-1 \text{ and } i=1,\dots,d
\end{align*}
Then this would contradicts $\text{span}(\cB ) =\reals^{pd}$.

Indeed, since $\absval{\cB}=pd$, it suffices to show that $\cB$ forms an independent set of vectors. To this end, we first express $(\bE M)^k E_p\be_i $ in a more convenient way. Note that,
\begin{align*}
M &= \mymat{0_d & I_d &&& \\&0_d & I_d&&\\ &&&&\\ && \ddots&\ddots\\&&& \\&&& 0_d & I_d\\C_{0}&&\dots&C_{p-2}&C_{p-1} }\\
&= \underbrace{\mymat{0_d\\0_d\\\vdots\\0_d\\I_d}  \mymat{C_0&C_1&\dots&C_{p-1}}}_{\cG} + 
\underbrace{\mymat{0_d & I_d &&& \\&0_d & I_d&&\\ &&&&\\ && \ddots&\ddots\\&&& \\&&& 0_d & I_d\\0_d&&\dots&0_d&0_d }}_O
\end{align*}
Hence, for any $\bx\in\reals^d$,
\begin{align}
M^K E_p \bx &=
\circpar{O^K+\sum_{k=1}^K \binom{K}{k} \cG^k  O^{K-k}}E_p \bx \nonumber\\
&=O^K E_p \bx +\sum_{k=1}^K \binom{K}{k}  \cG^k  O^{K-k}E_p \bx \nonumber\\
&=
\mymat{
0\\  \vdots \\ 0\\ \bx \\ \star\\\vdots\\ \star
}\
\begin{array}{c}
\\\\ \text{(} p-K \text{)'th position}  \\ \\\\ 
\end{array}
\end{align}
Now, for any linear combination of the elements of $\cB$  which equals zero, the coefficients which correspond to $M^{p-1}E_p e_i$ must vanish. This is due to the fact that the first $d$ coordinates of these vectors are $e_1,\dots,e_d$ which form a set of independent set. This, in turn ,yields that the coefficients which correspond to  $M^{p-2}E_p e_i$ must vanish, and so on. Thus, show that $\cB$ is an independent set.
\end{proof}

Hence, combining \ineqref{ineq:conv_lb} with \thmref{thm:cli_ex} we have proven that there exists $\bx^0\in\reals$ such that for sufficiently large $k$ it holds that,
\begin{align} \label{ineq:x_lb}
\norm{\bE \left[\bx^k\right]-\bx^*} 
&\ge \frac{1 }{\norm{U}}\norm{\bE \left[\bz^k\right]-\bz^*} \\
&\ge \frac{c_\cM k^m }{\norm{U}}\rho(M)^k\norm{\bz^0-\bz^*} \\
&\ge \frac{c_\cM k^m }{\norm{U}\norm{U^\top}}\rho(M)^k\norm{\bx^0-\bx^*} 
\end{align}

We conclude this section by the stating a corollary which summarizes Inequalities \ref{ineq:x_up} and \ref{ineq:x_lb}.
\begin{corollary} \label{cor:conv_rate_ulb}
Let $\cA$ be a $p$-CLI optimization algorithm over $\reals^d$.\\
Then, for any quadratic function $(A,\bb)\in\posdefun{d}{\Sigma}$, there exists an initialization point $x^0\in\reals^d$, such that 
\begin{align*}
\norm{\bE \left[\bx^k\right]-\bx^*} = \Theta\circpar{k^m\rho(M(A))^k\norm{\bx^0-\bx^*}}
\end{align*}
where $\rho(M(A))$ denotes the spectral radius of the iteration matrix $M(A)$, and  $m$ denotes the largest index among the eigenvalues of $M(A)$ whose absolute value is maximal.
\end{corollary}

\section{Iteration Complexity}
We saw that the convergence rate of $p$-CLI optimization algorithms is completely characterized by the spectral radius of the iteration matrix. 
This suggests that the spectral radius is an absolutely adequate way for comparing the performance of two different $p$-CLI optimization algorithms.
However, since many algorithms obtain sub-linear or super-linear (e.g.  quadratic) convergence rates, this measure lacks the generality necessary for comparing any iterative optimization algorithm. Thus motivating the following more conventional way for measuring performances of optimization algorithms.\\
\\
Let $\cA$ be a $p$-CLI optimization algorithm. 
For any $(A,\bb)$ we denote by\\ $\IC\circpar{\cA,\epsilon,(A,\bb),\bx^0}$ the minimal number of iterations $K$ required for algorithm $\cA$ to obtain 
\begin{align*}
\norm{\bE [\bx^k] - \bx^* } <\epsilon,\quad \forall k\ge K
\end{align*}
when applied to the function $(A,\bb)$ and initialized with $x^0$.\\ The \emph{Iteration~Complexity} of $\cA$ is defined by
\begin{align*}
\ICA = \max_{ (A,\bb) \in \posdefun{d}{\Sigma},~ x^0\in\reals^d } \IC\circpar{\cA,\epsilon,(A,\bb),\bx^0}
\end{align*}
This complexity is sometimes called \emph{Query Complexity}. 

The following theorem relates the bound regarding the spectral radius of the iteration matrix of any $p$-CLI optimization algorithm with its iteration complexity. 

\begin{theorem} \label{thm:ic_cli}
Fix $d\in\bN$ and let $\cA$ be a $p$-CLI optimization algorithm and $(A,\bb)\in\posdefun{d}{\Sigma}$ a quadratic function. Then,
\begin{align*}
\ICA =\Omega\circpar{\frac{\rho}{1-\rho}\ln(1/\epsilon)}
\end{align*}
and
\begin{align*}
\ICA=\bigO{\frac{1}{1-\rho}\ln(1/\epsilon)}
\end{align*}

where $\rho$ denotes the spectral radius of the iteration matrix $M(A)$.
\end{theorem}  
This theorem is easily proven be applying \lemref{lemma:spec_to_rate} (see Appendix \ref{chapter:tech}) on the bound given in \corref{cor:conv_rate_ulb}.\\

An arguably striking property of this bound is that it does not depend on the dimension $d$ of the problem space. Nevertheless, large $d$ might lead to a large number of iterations necessary for arriving at the asymptotic conditions.

\section{Coefficient Matrices of \texorpdfstring{$p$}{p}-CLI Optimization Algorithms}
We now know that the most significant factor which governs the iteration complexity of any $p$-CLI optimization algorithm is the spectral radius of its iteration matrix. We will see shortly that the spectrum of the iteration matrix, in particular its spectral radius, is highly related to the structure of the inversion matrix. To be specific, by closely inspecting the iteration matrix and the inversion matrix of FGD,AGD and HB one finds that the sum of the corresponding coefficient matrices always sum up to $I+N(A)A$. In the sequel we show that this is not a mere coincidence, but an extremely useful fact which forms our main tool in deriving lower bounds as well as $p$-CLI optimization algorithms. \\

Let $\cA$ be a $p$-CLI optimization algorithm over $\reals^d$ and let us denote its iteration matrix by $M(A)$, its inversion matrix by $N(A)$ and its lifting matrix by $U$. Furthermore, let $(A,\bb)\in\posdefun{d}{\Sigma}$ be a quadratic function. 
Recall that by  \thmref{thm:conv_cli} we have $$\bE \bz^k \to\circpar{I-\bE M(A)}^{-1}\bE \left[ U N(A)\bb\right] $$ for any initialization  point.\\
Combining this with \eqref{eq:conv_sol}, we get
\begin{align*}
U^\top \circpar{I-\bE M(A)}^{-1}\bE \left[U N(A)\bb\right] = - A^{-1}b
\end{align*}
Since this holds true for any $b\in\reals^d$, we get
\begin{align*}
U^\top \circpar{I-\bE M(A)}^{-1} U \bE  \left[N(A)\right] = - A^{-1}
\end{align*}
Clearly,  \eqref{eq:conv_sol} implies that $\bE N(A)$ must be invertible.  This yields,
\begin{align} \label{eq:conv_correct}
U^\top \circpar{I-\bE M(A)}^{-1} U = - (\bE \left[N(A)\right] A) ^{-1} 
\end{align}
This allows us to derive the following theorem regarding the sum of coefficient matrices of $p$-CLI optimization algorithms.

\begin{theorem} \label{thm:conv_correct}
Let $\cA$ be a $p$-CLI method and assume that its initialization point and update rule are as in \ref {iteratio_d},\ref{iterstep}, respectively, and let $A\in\posdefun{d}{\reals^{++}}$. Then, $\cA$ converges to $-A^{-1}\bb$ for any $\bb$ if and only if  
\begin{align} \label{eq:coef_formula}
\sum_{j=0}^{p-1} C_i(A) = I_d + N(A)A 
\end{align}
where $C_i(A)$ are the coefficient matrices which correspond to the iteration matrix $M(A)$ of $\cA$.\\
In which case, since the coefficient matrices are further assumed to be simultaneously triangularizable as in \thmref{thm:char_poly}, i.e. there exists an invertible matrix $Q\in\reals^{d\times d}$ such that 
\begin{align*}
	 T_i(A)\eqdef  Q^{-1}C_i(A)Q \quad i=0,1,\dots,p-1
\end{align*}
are upper triangular matrices, then 
\begin{align*}
\left\{ \left. \sum_{i=0}^{p-1} (T(A)_i)_{jj} ~\right|~ j\in[d] \right\}=\spec{I+NA}
\end{align*}
\end{theorem}

\begin{proof}
For the sake of convenience we shall omit the functional dependency of the iteration matrix and the inversion matrix on $A$. We show that the coefficient matrices of a $p$-CLI optimization method must obey \eqref{eq:coef_formula}, the other direction is proven by reversing the steps of the proof.\\
Recall that,
\begin{align*}
M = \mymat{0_d & I_d &&& \\&0_d & I_d&&\\ &&&&\\ && \ddots&\ddots\\&&& \\&&& 0_d & I_d\\C_{0}&&\dots&C_{p-2}&C_{p-1} }
\end{align*} 
for some $C_0,\dots,C_{p-1}\in\reals^d$. We define a partition of $M$ in the following manner,
\begin{align*}
\circpar{\begin{array}{c|cccc} M_{11}  &M_{12}\\ \hline M_{21}& M_{22} \end{array}} 
\eqdef\circpar{\begin{array}{cccc|c} 0_d & I_d &&& \\&0_d & I_d&&\\ &&&&\\ && \ddots&\ddots\\&&& \\&&& 0_d & I_d\\ \hline C_{0}&&\dots&C_{p-2}&C_{p-1} \end{array}}
\end{align*}
\eqref{eq:conv_correct} together with the hypothesis regarding $U$ readily implies,
\begin{align}
(I-M_{22} - M_{21}(I-M_{11})^{-1}M_{12})^{-1}   &= -(NA)^{-1}\nonumber\\
I-M_{22} - M_{21}(I-M_{11})^{-1}M_{12}   &= -NA \nonumber\\
M_{22} + M_{21}(I-M_{11})^{-1}M_{12}   &= I+NA \label{eqst1} 
\end{align}
This equation may be easily derived using Schur Complement. Moreover, it is straightforward to verify that,
\begin{align*}
\circpar{I-M_{11}}^{-1} = \mymat{I_d && I_d&I_d\\&I_d && I_d\\&&\ddots \\&&&I_d  }
\end{align*}
Plugging in this into (\ref{eqst1}) yields
\begin{align} 
\sum_{i=0}^{p-1} C_i = I+NA
\end{align}

Consequently,
\begin{align*}
\left\{ \left. \sum_{i=0}^{p-1} (T_i)_{jj} ~\right|~ j\in[d] \right\}&=\spec{\sum_{i=0}^{p-1} T_i} \\
&= \spec{Q\circpar{\sum_{i=0}^{p-1} T_i}Q^{-1}} \\
&= \spec{\sum_{i=0}^{p-1} QT_iQ^{-1}} \\
&= \spec{\sum_{i=0}^{p-1} C_i} \\
&=\spec{I+NA}
\end{align*}
Thus concludes the proof.

\end{proof}

The results regarding $p$-CLI optimization algorithms which has been  presented in the last section may be succinctly summarized by the following useful corollary which combines the implications \thmref{thm:char_poly} and \thmref{thm:conv_correct}. This corollary will be used extensively throughout the next chapters. 
\begin{corollary} \label{cor:ess_for}
The characteristic polynomial of any $p$-CLI optimization algorithm, applied on some quadratic function $(A,\bb)\in\posdefun{d}{\Sigma}$, is a product of $d$ monic polynomials $s_1(\lambda),\dots,s_d(\lambda)$ of degree $p$, for which
it holds that
\begin{align*}
\spec{-N(A)A} = \{s_1(1),s_2(1),\dots,s_d(1)\} 
\end{align*}

\end{corollary}

\chapter{Lower Bounds on the Iteration Complexity of \texorpdfstring{$p$}{p}-CLI Optimization Algorithms} \label{chapter:lower_bounds}

In this chapter we derive lower bounds on the iteration complexity of $p$-CLI optimization algorithms. The technique used for accomplishing this goal is discussed in detail in the following.\\

Let $\cA$ be a $p$-CLI optimization algorithm and let us denote its iteration matrix and inversion matrix by $M(A)$ and $N(A)$, respectively.
Recall that we concluded the last chapter by showing that there exist $d$ monic polynomials $s_1(\lambda),\dots,s_d(\lambda)$ of degree $p$, such that
\begin{align} \label{eq:poly_and_spec}
\spec{-N(A)A} &= \{s_1(1),s_2(1),\dots,s_d(1)\} \\
\spec{M(A)} &= \bigcup_{i=1}^d \spec{s_j(z)}
\end{align}
for any $(A,b)\in\posdefun{d}{\Sigma}$.\\
Consequently, one might ask what can be deduced about the spectral radius (see \defref{eq:poly_spec}) of some $p$-degree monic polynomial $q(z)$, for which  $q(1)$ is fixed and given? \\

In the following section we give a partial, yet quite satisfactory, answer for this question by using elementary properties of polynomials. Loosely speaking, we show that if $q(1)$ is much smaller or much bigger than one, then the spectral radius must be relatively big. We are therefore led to the following question: How far from one is the spectrum  of $-N(A)A$ ? Not too surprisingly,  the answer is: It depends! \\
\\
Indeed, we saw in section \ref{section_spec_algo} that  the inversion matrix of Newton method is define by
\begin{align*}
N(A)=-A^{-1}
\end{align*}
From which we conclude that  $\spec{-N(A)A}$ is exactly $\{1\} $.   \\
Nevertheless, the computational complexity of inverting $d\times d$ regular matrices is at least of an order $\Omega(d^2)$ (one must at least read all the entries of $A$), thus rendering  impractical methods which use this operation for large $d$.\\
That being said, if $N(A)$ does not approximate $-A^{-1}$ well, then the spectral radius of the iteration matrix $\rhoM$ would be large, e.g. $N(A)=0$ for which $\spec{-N(A)A}=\{0\}$. This is the reason why we decided to call $N(A)$ inversion matrix. As a matter of fact, many optimization algorithms can be seen as different strategies as to how should one balance the computational labor between two tasks: Computing an approximated inverse $N(A)$ for a given $d\times d$ matrix $A$ vs. executing many iterations of the algorithms.\\

The technique should be clear by now: various structural limitation imposed on the inversion matrix induces constraints on $\spec{-N(A)A}$, which in turn implies lower bounds on the spectral radius of the iteration matrix. In the sequel we will be inspecting very simple structures, such as scalar matrices and diagonal ones. We will assume that $\Sigma=[\mu,L]$ for some $L>\mu>0$ throughout this chapter. It is important to note that these results are likely to be obtained for a wider family of inversion matrices,  such as inversion matrices whose each coordinate depends on a relatively small set of entries of $A$.\\

After deriving lower bounds on the spectral radius of iteration matrix we will establish lower bound on the iteration complexity of $p$-CLI optimization algorithms. Lastly, we address the question whether this bound is tightness.

\section{Lower Bound on the Spectral Radius of Monic Polynomials} \label{section:eco_poly}
In the sequel we tackle the question mentioned in the discussion above: what can be said about the spectral radius of a monic polynomial $q(x)$ of degree $p$, given its evaluation in $x=1$ is some real scalar $r$. \\
We may first formalize this question in the following slightly more general minimization task
\begin{align} \label{opt:comp_poly_min}
\argmin\left\{ \rho(q(z))~ \left|~ q\in\reals[z], q(z) \text{ is monic}, \deg(q) =  p \text{ and } q(1) = r \right.\right\}
\end{align}
Furthermore, given $r\ge0$, let us define the following polynomial  
\begin{align} \label{def:comp_poly}
q_r^*(z) \eqdef \circpar{z-(1-\sqrt[p]{r})}^p
\end{align}
In chapter \ref{chapter:new_algo} this simple polynomials will have a major role in designing efficient algorithms due to its very attractive property - For any $r\ge0$, the polynomial $q_r^*(z)$ is the unique minimizer of \ref{opt:comp_poly_min} (Clearly, $q_r^*(z)$ is of degree $p$ and $q_r^*(1)=r$). As matter of fact, the following theorem states a slightly stronger result.

\begin{lemma} \label{lem:eco_poly}
Suppose $q(z)$ is a monic polynomial of degree $p$ with complex coefficients. Then,
\begin{align*}
\rho(q(z))\le \absval{\sqrt[p]{\absval{q(1)}}-1} \iff q(z)=q_r^*(z)
\end{align*}
\end{lemma}

\begin{proof}
The 'if' implication is obvious. Let us prove the 'only if' part.\\
By the fundamental theorem of algebra $q(z)$ has $p$ roots. Let us denote these roots by $\zeta_1,\zeta_2,\dots,\zeta_p$ . Equivalently,
\begin{align*}
q(z) = \prod_{i=1}^p (z-\zeta_i)
\end{align*}
Let us denote $r\eqdef\absval{q(1)}$. If $r\ge 1$ we get 
\begin{align} \label{ineq:ttt}
r &= \absval{\prod_{i=1}^p (1-\zeta_i)} = \prod_{i=1}^p \absval{1-\zeta_i} \le \prod_{i=1}^p (1+\absval{\zeta_i}) \nonumber\\&\le \prod_{i=1}^p (1+\absval{\sqrt[p]{r}-1}) 
= \prod_{i=1}^p (1+ \sqrt[p]{r}-1) =  r
\end{align}
Consequently, \ineqref{ineq:ttt} is, in fact, equality. Therefore, 
\begin{align} \label{eq:eq_roots_poly}
\absval{1-\zeta_i}=1+\absval{\zeta_i}=\sqrt[p]{r},\quad \forall i\in[p]
\end{align}
Now, for any two complex numbers $w,z\in\bC$ it holds that 
\begin{align*}
\absval{w+z}=\absval{w}+\absval{z} \iff \text{Arg}(w)=\text{Arg}(z)
\end{align*}
Using this fact in the first equality of \eqref{eq:eq_roots_poly}, we get that $\text{Arg}(-\zeta_i)=\text{Arg}(1)=0$, i.e., $\zeta_i$ are negative real numbers. Writing $-\zeta_i$ in the second equality of  \eqref{eq:eq_roots_poly} instead of $\absval{\zeta_i}$, yields $1-\zeta_i=\sqrt[p]{r}$. Thus, concludes this part of the proof.\\

The proof for $r\in[0,1)$ follows along the same lines by employing the reverse triangle inequality,
\begin{align*}
r &= \prod_{i=1}^p\absval{ 1-\zeta_i} \ge 
\prod_{i=1}^p \circpar{1 - \absval{ \zeta_i} }
\ge 
\prod_{i=1}^p \circpar{1 - \absval{ \sqrt[p]{r}-1} }\\
&= \prod_{i=1}^p \circpar{1 - (1- \sqrt[p]{r})}=r 
\end{align*}
Note that we implicitly used the fact that $r\in[0,1)\implies \absval{\zeta_i}\le1$ for all $i$.
\end{proof}

Thus, we have shown that, in a certain sense, the spectral content of $q_r^*(z)$ is  the most 'economic'. Indeed,   for $r\ge0$ the spectral radius of any other monic polynomial of the same degree whose value at $z=1$ equals $q_r^*(1)$, must be at least the spectral radius of $q_r^*(z)$, namely, $\rho(q(z))\ge \absval{\!\sqrt[p]{r}-1\!}$.\\
For real monic polynomials, if $r\le0$ then the mere fact that
\begin{align}
\lim_{z\in\reals,  z\to\infty} q(z) = \infty 
\end{align}
together with the Mean-Value theorem implies that $\rho(q(z))\ge1$, which is sufficiently big for our purposes. The following corollary summarizes the this observation together with last Lemma.

\begin{corollary} \label{cor:comp_poly}
Let $q(z)$ be a monic polynomial of degree $p$.\\
If $q(1)\ge0$, then $$\rho(q(z))\ge\absval{\!\sqrt[q]{q(1)}-1\!}$$ otherwise if $q(1)<0$ , and if in addition the coefficients of $q(z)$ are real, then, $$\rho(q(z))>1$$
These lower bounds are tight and attainable by $q_r^*(z)$  defined in \ref{def:comp_poly}.
\end{corollary}

\begin{figure}[h]
  \centering
     \includegraphics[width=0.7\textwidth]{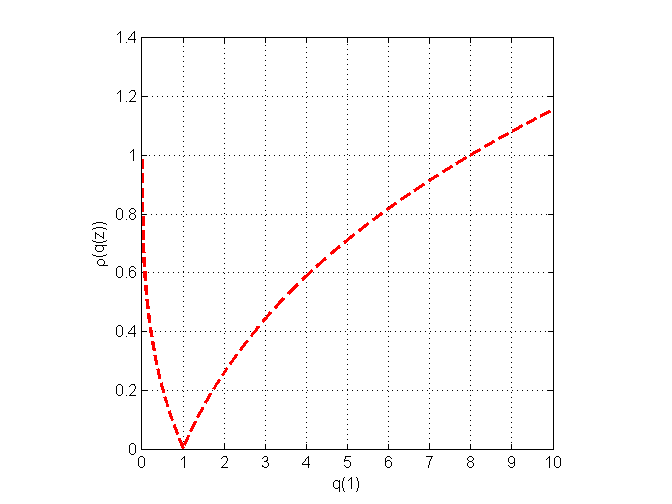}
  \caption{A lower bound on the spectral radius of any monic polynomial $q(z)$ of degree 3, based on its evaluation at $z=1$.}
\end{figure}

\begin{remark}
The requirement that the coefficients of $q(z)$ should be real  is inevitable.  To see why, consider the following polynomial, 
\begin{align*}
u(z)  = \circpar{z-\circpar{1-0.5e^{\frac{i\pi}{3}}}}^3
\end{align*}
Although $u(1)=\circpar{1-\circpar{1-1/2e^{\frac{i\pi}{3}}}}^3 = -1/8\le0$, it holds that $\rho(u(1))<1$. the moduli of all roots is strictly smaller than $1$. Indeed, not all the coefficients of $u$ are real. Notice that the claim does hold for degree $\le3$, regardless of the additional assumption on the coefficients of $u(z)$.
\end{remark}

\section{Scalar Inversion Matrix} \label{section_non_dep}
Let $\cA$ be a $p$-CLI algorithms, let $M(A)$ be its iteration matrix and $N(A)$ be its inversion matrix. Assume further that $N(A)$ is a scalar matrix, i.e., $N(A)=\nu(A)I_d$ for some scalar-valued function $\nu(A)$.\\
In the sequel we show a lower bound on the spectral radius of the iteration matrix of $\cA$. Let $B\in\posdef{d}{[\mu,L]}$ be a positive definite matrix, such that $\mu,L\in\spec{B}$, e.g. $\diag{L,\mu,\dots,\mu}$. The condition number $Q$ of any such matrix is exactly $L/\mu$. Let us denote $\nu\eqdef\nu(B)$.  Clearly, $$-\nu\{\mu,L\} \subseteq \spec{-\nu B} = \spec{-N(B)B}$$
Thus,  by \eqref{eq:poly_and_spec}  one concludes that there exist two real monic polynomials $s_\mu(z),s_L(z)$, both of degree $p$, such that,
\begin{align} \label{eq:poly_at_one}
s_\mu(1) &= -\nu\mu\\
s_L(1) &= -\nu L
\end{align}
and such that 
\begin{align*}
\spec{s_\mu(z)}\bigcup\spec{s_L(z)}\subseteq\spec{M(B)}
\end{align*}
Hence
\begin{align}  \label{ineq:spec_two_poly} 
\rho(M(B)) \ge \max\left\{\rho(s_\mu(z)),\rho(s_L(z))\right\}
\end{align}
Establishing  \eqref{eq:poly_at_one}, we may now apply \corref{cor:comp_poly} on $s_\mu(z),s_L(z)$ to get an explicit lower bound from the last inequality.\\
\\
Note that if $\nu\ge0$ then $s_\mu(1),s_L(1)\le 0$, thus by \corref{cor:comp_poly} 
\begin{align*}
\rho(s_\mu(z)),\rho(s_L(z))\ge1
\end{align*}
Plugging it in \ineqref{ineq:spec_two_poly} yields,
\begin{align*}
\rho(M(B))\ge 1
\end{align*}
Likewise, if $\nu \le \frac{-2^p}{L}$  then again by \corref{cor:comp_poly} 
\begin{align*}
\rho(s_L(z))\ge \sqrt[p]{-\nu L}-1\ge1
\implies \rho(M(B))\ge1
\end{align*}
Both ranges violate the assumption regarding the convergence of $\cA$. Thus, from now on we may assume  that $\nu\in(0,\frac{-2^p}{L})$.\\

We split the rest of the proof into three different cases:
\begin{enumerate} 
\item $-1/L\le\nu<0$ - In which case,
\begin{align*}
&\frac{-\mu}{L}\le \nu\mu\\
\implies&
\frac{\mu}{L}\ge -\nu\mu\\
\implies&
\sqrt[p]{\frac{\mu}{L}} \ge \sqrt[p]{-\nu\mu}\\
\implies&
-\sqrt[p]{\frac{\mu}{L}} \le -\sqrt[p]{-\nu\mu}\\
\implies&
1-\sqrt[p]{\frac{\mu}{L}} \le 1-\sqrt[p]{-\nu\mu}
\end{align*}
Therefore
\begin{align*}
\rho(M(B))&\ge\rho(s_\mu(z)) \ge 1- \sqrt[p]{-\nu \mu}  \ge 
 1-\sqrt[p]{\frac{\mu}{L}} = \frac{ \sqrt[p]{Q} -1 }{\sqrt[p]{Q}} 
\end{align*}
When $\nu=-1/L$ then the last inequality becomes tight.
	
\item $-1/\mu < \nu <-1/L$ -  By \ineqref{ineq:spec_two_poly}
\begin{align} \label{ineq:spec_range_one}
\rho(M) &\ge \max \left\{ \sqrt[p]{-\nu L} -1, 1-\sqrt[p]{-\nu \mu} \right\} \nonumber\\
 &\ge \min_{\nu' \in (-1/L, -1/\mu)}\max \left\{ \sqrt[p]{-\nu' L} -1, 1-\sqrt[p]{-\nu' \mu} \right\} 
\end{align}
Since $\sqrt[p]{-\nu L} -1$ monotonically decreases with $\nu$ and $1-\sqrt[p]{-\nu \mu}$ monotonically increases with $\nu$, the minimizer of the last expression $\nu^*$
is attained when both arguments are equal (for otherwise, we could make $\nu$ slightly bigger$\backslash$smaller).\\
Thus,
\begin{align} \label{eq:con_lb_minimizer}
&\sqrt[p]{-\nu^* L} -1= 1-\sqrt[p]{-\nu^* \mu}\nonumber\\
\implies& \nu^* = -\circpar{\frac2{\sqrt[p]{ L} +\sqrt[p]{\mu} }}^p 
\end{align}
Note that $-\nu^*$ interpolates the harmonic mean of $\frac{1}{L}$ and $\frac{1}{\mu}$ with their geometric mean as $p\to\infty$. Now, plugging in $\nu^*$ into \ineqref{ineq:spec_range_one} yields,
\begin{align*}
&\rho(M) \ge 1- \sqrt[p]{\circpar{\frac2{\sqrt[p]{ L} +\sqrt[p]{\mu} } }^p\mu} \\
\implies& \rho(M) \ge  \frac{\sqrt[p]{ L} -\sqrt[p]{\mu} }{\sqrt[p]{ L} +\sqrt[p]{\mu} }= \frac{\sqrt[p]{ Q} -1 }{\sqrt[p]{Q} +1}
\end{align*}

\item $\frac{-2^p}{L}<\nu  \le -1/\mu $ - First, note that in order for this range to be valid, one must assume that $2^p>\frac{L}{\mu}=Q$. Now, since\\
\begin{align*}
		\frac{-1}{\mu}\ge \nu 
\implies \sqrt[p]{\frac{L}{\mu}}-1\le \sqrt[p]{-\nu L}-1
\end{align*}
Hence,
\begin{align*}
\rho(M(B)) &\ge  \rho(s_L(z)) \ge \sqrt[p]{-\nu L}  -1 \ge \sqrt[p]{\frac{L}{\mu}}-1 = \sqrt[p]{Q}-1
\end{align*}

\end{enumerate}
To conclude, regardless of the actual value of $\nu$ we have,
\begin{align*}
\rho(M(B))\ge \min\left\{1,\frac{\sqrt[p]{Q}-1}{\sqrt[p]{Q}} ,\frac{\sqrt[p]{Q}-1}{\sqrt[p]{Q}+1},\sqrt[p]{Q}-1 \right\} = \frac{\sqrt[p]{Q}-1}{\sqrt[p]{Q}+1}
\end{align*}
Thus, the convergence rate of any $p$-CLI method whose inversion matrix is scalar must be bigger or equal to 
\begin{align*}
\frac{\sqrt[p]{Q}-1}{\sqrt[p]{Q}+1}
\end{align*}
when applied on $B$ (see \corref{cor:lb_ic_explicit} for the analogous lower bound on the iteration complexity).

\section{Diagonal Inversion Matrix}
In the sequel we prove that allowing a more general structure of inversion matrices, namely diagonal matrices, does not necessarily lead to a faster convergence rate. In particular, we show that for any $p$-CLI optimization algorithm $\cA$ whose inversion matrix is diagonal there exists a quadratic function for which $\cA$ obtains a convergence rate which is no better than the rate obtained by $p$-CLI optimization algorithms with scalar inversion matrix, namely 
\begin{align*}
\frac{\sqrt[p]{Q}-1}{\sqrt[p]{Q} +1}
\end{align*}

Let $\cA$ be a $p$-CLI optimization algorithm, let $M(A)$ and $N(A)$ be its iteration matrix and its inversion matrix, respectively. Assume further that $N(A)$ is a diagonal matrix. Let 
\begin{align}
B = \mymat{ \frac{L+\mu}{2} & \frac{L-\mu}{2} \\ \frac{L-\mu}{2} & \frac{L+\mu}{2}}
\end{align}
note that $B$ is positive definite matrix and that $\spec{B}=\set{\mu,L}$. We wish to derive a lower bound on $\rho(M(B))$. To this end, denote 
\begin{align*}
N\eqdef N(B) = \mymat{\alpha & 0 \\ 0 & \beta}
\end{align*}
where $ \alpha,\beta\in\reals$ and denote the eigenvalues of $-NB$ by $\lambda_1(\alpha,\beta), \lambda_2(\alpha,\beta)  $. \\
Following similar arguments to the scalar case, we have that both eigenvalues of $-NB$ must be strictly positive. Furthermore, since $\cA$ is assumed to be convergent, we have $\rho(M)<1$, which together with \corref{cor:comp_poly} and \ineqref{ineq:spec_two_poly} implies 
\begin{align} \label{ineq:spec_diagonal}
\rho(M) \ge \min_{\alpha,\beta}\max\left\{ \absval{\sqrt[p]{\lambda_1(\alpha,\beta)}-1},\absval{\sqrt[p]{\lambda_2(\alpha,\beta) }-1}\right\}
\end{align}
A straightforward calculation shows that the eigenvalue of $-NB$ are
\begin{align*} \label{eq:eigs_of_diagonal_inversion}
\lambda_1(\alpha,\beta), \lambda_2(\alpha,\beta)   &=  \frac{-(\alpha+\beta)(L+\mu)}{4} \pm
\sqrt{ \circpar{\frac{(\alpha+\beta)(L+\mu)}{4}}^2 -  \alpha\beta L\mu}\\
&=\frac{-(\alpha+\beta)(L+\mu)}{4} \pm
\sqrt{ (\alpha+\beta)^2 \frac{(L-\mu)^2}{16} +  \frac{1}{4} (\alpha-\beta)^2  L\mu}
\end{align*}
Thus showing that for any $\alpha,\beta$, it holds for $\nu=\frac{\alpha+\beta}{2}$ that 
\begin{align*}
\max\left\{ \absval{\sqrt[p]{\lambda_1(\alpha,\beta)}-1},\absval{\sqrt[p]{\lambda_2(\alpha,\beta) }-1}\right\} &\ge 
\max\left\{ \absval{\sqrt[p]{\lambda_1(\nu,\nu)}-1},\absval{\sqrt[p]{\lambda_2(\nu,\nu) }-1}\right\}\\
&=
\max\left\{ \absval{\sqrt[p]{-\nu \mu}-1},\absval{\sqrt[p]{-\nu L }-1}\right\}
\end{align*}
Therefore, \ineqref{ineq:spec_diagonal} becomes
\begin{align} 
\rho(M) \ge \min_{\gamma}\max\left\{ \absval{\sqrt[p]{\lambda_1(\gamma,\gamma)}-1},\absval{\sqrt[p]{\lambda_2(\gamma,\gamma) }-1}\right\}
\end{align}
Now, according to the very same analysis done in the scalar case, it holds that 
\begin{align*}
\rho(M(B)) \ge \frac{\sqrt[p]{Q}-1}{\sqrt[p]{Q}+1}
\end{align*}
Note that although we have proven the claim for $d=2$. The general case follows easily by embedding $B$ as a principal sub-matrix in some matrix of higher dimension.\\

\section{Is This Lower Bound Tight?} \label{section:is_this_tight}
The previous sections may be briefly stated by the following.
\begin{corollary} \label{cor:lb_ic_explicit}
Let $\cA$ be a $p$-CLI optimization algorithm for $\posdefun{d}{[\mu,L]}~(L>\mu>0)$. If the inversion matrix of $\cA$ is diagonal, then there exists a $B\in\posdefun{d}{[\mu,L]}$ such that 
\begin{align} \label{ineq:lb_conv_rate22}
\rho(M(B)) \ge \frac{\sqrt[p]{Q}-1}{\sqrt[p]{Q}+1}
\end{align}
where $M(B)$ denotes the corresponding iteration matrix.\\
In particular, by \thmref{thm:ic_cli} 
\begin{align} \label{eq:IC_lb}
\ICA = \Omega\circpar{\frac{\sqrt[p]{Q}-1}{2}\ln(1/\varepsilon)}
\end{align}

\end{corollary}

Well, is this lower bound tight? the answer heavily depends on how rich the structure of the coefficient matrices is.\

For $p=1$ this bound is attained by FGD (see \ref{spec:fgd}). For $p=2$ things are getting much more interesting. It can be shown that the spectral radius of the Heavy Ball method is,
\begin{align} \label{eq:spec_hb}
\Theta\circpar{\frac{\sqrt[2]{Q}-1}{\sqrt[2]{Q}+1}}
\end{align}
Hence, by \thmref{thm:ic_cli} 
\begin{align} \label{eq:opt_IC_lb}
\IC_{[\mu,L]}\circpar{\text{HB},\epsilon} = \bigO{\frac{\sqrt[2]{Q}-1}{2}\ln(1/\varepsilon)}
\end{align}
Although HB obtains the best possible iteration complexity in its class, it has a major disadvantage, namely, when applied on general smooth strongly-convex functions, one can guarantee local convergence only (see Section 3.2.1 in \cite{polyak1987introduction}).  
Another disadvantage is that it is not clear a-priori when does the phase of linear convergence with factor as in \eqref{eq:spec_hb} starts. In contrast to this, AGD attains a linear convergence with a slightly worse factor of 
\begin{align} \label{eq:spec_agd}
\Theta\circpar{\frac{\sqrt{Q}-1}{\sqrt{Q}}}
\end{align}
for any smooth strongly-convex function, start from the very first iteration. In which case the iteration complexity is
\begin{align}
\IC_{[\mu,L]}\circpar{\text{AGD},\epsilon} = \bigO{\circpar{\sqrt{Q}-1}\ln(1/\varepsilon)}
\end{align}

Furthermore, recall that both the Heavy Ball method and AGD has a scalar inversion matrix
\begin{align*}
N_{\text{HB}} &= \frac{-4}{\circpar{\sqrt{L}+\sqrt{\mu}}^2}I_d\\
N_{\text{AGD}} &= \frac{-1}{L}I_d 
\end{align*}
Surprisingly enough, these constants perfectly match the optimal constant in the first range $[-1/L,0)$ and in the second range $(-1/\mu,-1/L)$ of $\nu I$, as shown in section (\ref{section_non_dep}). In chapter \ref{section_schemes} we will show an even more surprising result, namely, how can one recover AGD, the Heavy Ball method and perhaps other optimization algorithms by carefully choosing $\nu$ and a matching iteration matrix. \\

We further remark that the lower bound mentioned in \cite{nesterov2004introductory} is $$\frac{\sqrt{Q}-1}{4}\ln(1/\varepsilon)$$
Hence, the lower bound shown in \ref{eq:IC_lb} is slightly better. More importantly, the latter lower bound holds for sufficiently large number of iterations, as opposed to the lower bound in \cite{nesterov2004introductory}  which is true for $\bigO{d}$ number of iterations. That being said, \ref{eq:IC_lb} is proved to be valid for a much more restricted family of algorithms. Put differently, both lower bounds combined show that for essentially any number of iterations, and for a relatively wide family of algorithms the iteration complexity of $2$-CLI optimization methods is  
\begin{align*}
\Theta\circpar{\sqrt{Q}\ln(1/\epsilon)}
\end{align*}

What about $p>2$? In section \ref{section:new_algo} we will see that, as long as no restriction is imposed on the coefficients matrices, the lower bound presented in \ineqref{ineq:lb_conv_rate22} is tight, i.e., for any $p\in\bN$ there exists a $p$-CLI optimization algorithm with scalar inversion matrix whose convergence rate is exactly
\begin{align}
\Theta\circpar{\frac{\sqrt[p]{Q}-1}{\sqrt[p]{Q}+1}}
\end{align}
That being said, the disadvantage in this apparently ideal algorithm is that in order to execute it we need to have a good approximation of the hessian's spectrum of the quadratic function under consideration. This leads to coefficients matrices which are hard to compute and may cost more than the actual task of finding a minimizer. So, can this lower bound be obtained by $p$-CLI optimization algorithm whose inversion matrix is diagonal and whose iteration matrix can be efficiently computed? Before answering this question, one must define 'efficiently'. If by efficiently-computed iteration matrix we mean coefficient matrices which are linear in $A$, namely
\begin{align*}
C_k(A) &= \alpha A + \beta I_d,~k=0,1,\dots,p-1
\end{align*}
then we conjecture that this lower bound is not tight. This may be formally stated by the following. 
\begin{conjecture*} 
Let $\cA$ be a $p$-CLI optimization algorithm whsoe inversion matrix is diagonal and whose iteration matrix $M(A)$ consists of linear coefficient matrices, that is 
\begin{align*}
C_k(X) = \alpha_k X + \beta_k I_D
\end{align*}
for some real scalars $\alpha_k,\beta_k\in\reals$. Then 
\begin{align*}
\exists A\in\posdefun{d}{[\mu,L]},~\rho(M(A)) \ge \frac{\sqrt{Q}-1}{\sqrt{Q}+1}
\end{align*}
where $Q\eqdef L/\mu$.
\end{conjecture*}

This conjecture may be reduced to a pure question in the analytic thoery of polynomials,
\begin{conjecture*} [Polynomials] \label{conj:linear_coeff_poly}
For any $0<\mu<L$ and for any $p-1$ degree polynomials $a(z),b(z)$ such that $b(1)=1$ there exists $\eta\in[\mu,L]$ such that
\begin{align*}
\rho(z^p - (\eta a(z) + b(z))) &\ge \frac{\sqrt{L/\mu} -1}{\sqrt{L/\mu}+1}
\end{align*}
\end{conjecture*}

This conjecture is mainly supported by the following two observations: By the information-based lower bound mentioned in \ref{opt:sqrtlb}, we know that this conjecture holds true for the first $\bigO{d}$ iterations. Secondly, the conjecture had withstood a vast number of randomly drawn polynomials. The generalization of this conjecture to polynomial expressions of the coefficient matrices to any other degree is straightforward.

\chapter{Designing Optimization Algorithms using Optimal Pairing Schemes} 
\label{chapter:new_algo} 
\label{section_schemes}

Up to this point we have projected various optimization algorithms on the framework of $p$-CLI methods, thereby converting questions regarding convergence properties of these algorithms into questions regarding roots of polynomials. \\
In what follows, we shall head in the opposite direction. First, we design polynomials so as to meet a prescribed set of constraints. Then we derive the corresponding  optimization algorithms. Remarkably enough, by carefully adjusting these polynomials, we derive a novel optimization algorithm which allows better utilization of prior knowledge about quadratic functions as well as 're-discover' FGD, HB, AGD and other optimization algorithms, in an unconventional systematic way that will turn out be rather natural. This line of inquiry is particularly important due to the obscure  origins of AGD, and further emphasizes its algebraic nature.\\

\section{Definitions}
Let $p$ be lifting factor and $N(A)$ be an inversion matrix over $\reals^d$. Any iteration matrix $M(A)$ which, when paired with $N(A)$, yields a valid $p$-CLI optimization algorithm for any $A\in\posdefun{d}{[\mu,L]}$, is said to be a \emph{Matching Iteration Matrix} for $N(A)$. If in addition, $M(A)$ admits the minimal spectral radius among all other matching iteration matrices, then it is said to be an \emph{Optimal Iteration Matrix} for $N(A)$ and  $\circpar{M(A),N(A)}$  is said be an \emph{Optimal Pairing}.\\

Now, consider the following problem: Find an optimal iteration matrix for given lifting factor $p$ and inversion matrix $N(A)$ over $\reals^d$.\\

In the next section we show how to derive the optimal iteration matrix for a given scalar inversion matrix, given that no efficiency constraints are involved. As we mentioned before, the computational cost of obtaining such iteration matrix may be too high. Consequently, in Section \ref{section:lin_coeff_mat} we focus on a specific family of iteration matrices which can be efficiently computed, namely, iteration matrices whose coefficient matrices are linear expressions in $A$. Lastly, we show how FGD,AGD and HB form an instantiation of optimal pairing schemes with various scalar inversion matrices.

\section{Optimal Pairing Schemes of Scalar Inversion Matrices} \label{section:new_algo}
 In the sequel we shall closely inspect the proof of the lower bound shown in \ref{eq:IC_lb} so as to derive a $p$-CLI which admits the same convergence rate. \\
Let $\cA$ be a $p$-CLI optimization algorithm and let us denote its iteration matrix, inversion matrix and its coefficient matrices by $M(A)$, $N(A)$ and $C_0(A), C_1(A),\dots,C_{p-1}(A)$, respectively, and assume that $N(A)=\nu I$ for some real constant $\nu\le0$. Lastly, let $(A,b)\in\posdefun{d}{\Sigma}$ be a given quadratic function.\\
Recall that we showed in \corref{cor:ess_for} that there exist $d$ monic polynomials $s_1(\lambda),\dots,s_d(\lambda)$ of degree $p$, such that
\begin{align} 
\spec{-N(A)A} &= \{s_1(1),s_2(1),\dots,s_d(1)\} \\
\spec{M(A)} &= \bigcup_{i=1}^d \spec{s_j(z)}
\end{align}

It is further implied by \lemref{lem:eco_poly} that $s_j(\lambda)$  attain the lower bound on the spectral radius of polynomials if and only if 
\begin{align} \label{eq:eco_poly_coeff}
s_j(\lambda)=\circpar{\lambda-\circpar{1-\sqrt[p]{-\nu \mu_j}}}^p = \sum_{k=0}^p \binom{p}{k}\circpar{ \sqrt[p]{-\nu\mu_j} -1}^{p-k}  \lambda^{k}
\end{align}
where $\mu_j$ are the corresponding eigenvalues of $A$. Our proof of the lower bound had used the fact the $s_j(\lambda)$ are constrained to have a certain value at $\lambda=1$, by which we concluded that the spectral radius must be large enough. \\
Due to the crucial fact that the coefficients of $s_j(\lambda)$ are nothing but the eigenvalues of the coefficient matrices, ordered arbitrarily, we see that the lower bound is obtained if and only if the eigenvalues of the coefficient matrices perfectly match the coefficients of the polynomials in \eqref{eq:eco_poly_coeff}. The latter condition is easily satisfied as illustrated by the following construction.\\

Since $A$ is a positive definite matrix, there exists an unitary matrix $O$ such that $$\Lambda\eqdef O^\top A O$$ is diagonal. Evidently, the diagonal elements of $-\nu \Lambda$ are the eigenvalue of $-N(A)A$ in an arbitrary order. Let us define the $p$ diagonal matrices $D_0,D_1,\dots,D_{p-1}$ so as to respect \eqref{eq:eco_poly_coeff} 
\begin{align*}
D_k = \mymat{-\binom{p}{k}\circpar{ \sqrt[p]{-\nu \Lambda_{11}} -1}^{p-k} \\ & -\binom{p}{k}\circpar{ \sqrt[p]{-\nu \Lambda_{22}} -1}^{p-k}\\ && \ddots  \\ &&& -\binom{p}{k}\circpar{ \sqrt[p]{-\nu \Lambda_{dd}} -1}^{p-k} }
\end{align*}
The coefficient matrices are correspondingly  defined to be
\begin{align*}
C_k = O D_k  O^\top
\end{align*}
By applying \thmref{thm:conv_correct}, it can be verified  that in such case $s_j(\lambda)$ are consistent with \eqref{eq:eco_poly_coeff} and that for any $\nu\in(-2^p/L,0)$, the $p$-CLI optimization algorithm specified above converges to the minimizer of $\circpar{A,b}$ in a rate of 
\begin{align*}
\rho(M(A)) = \max_{j=1,\dots,d} \absval{ \sqrt[p]{-\nu \mu_j} - 1} 
\end{align*}
Finally, choosing  
\begin{align*}
\nu=-\circpar{\frac2{\sqrt[p]{ L} +\sqrt[p]{\mu} }}^p 
\end{align*}
as in \eqref{eq:con_lb_minimizer} results in the optimal $p$-CLI optimization algorithm for a given $p$ whose iteration complexity is (\ref{eq:opt_IC_lb}). Further assuming that the computational cost grows linearly with the lifting factor $p$, then one may obtain a running time of $\Theta(\ln(Q)\ln(1/\epsilon))$ by properly optimizing $p$. \\
The major drawback of this algorithm is that its implementation requires the spectral decomposition of $A$, which, if no other structural assumptions are imposed, is a harder task than computing the inverse of $A$. Thus, rendering this algorithm impractical in the general case, where good approximations to the spectrum of $A$ are  not in reach. Moreover, any other algorithm which achieve this lower bound exactly must have its coefficient matrices similar to $D_k$ for any $(A,b)\in\posdefun{d}{[\mu,L]}$. Therefore, it is fairly likely that these kind of algorithms would be inefficient in terms of computing the corresponding iteration matrix. \\

A comparison of the performance of various optimization algorithms, including the one discussed above, is provided in the following figure. We measure the accuracy of each algorithm versus the iteration number, where 'accuracy' is defined by  
\begin{align*}
\norm{  -A^{-1}b - U^\top z^k}
\end{align*}
($U$ is as defined in \ref{def:U_def})\\
for $2\times 2$ diagonal matrix $A$ whose diagonal elements are $\mu,L$ and $$b=\mymat{\mu\\L}$$ (Note that in this case $-A^{-1}b=-(1,1)^\top$). Each algorithm is initialized with $z^0=0$.\\
 As can be seen, these experiments reflect the theoretical results presented in the previous chapter, in particular they demonstrate the superiority of the optimization algorithm presented in this section over AGD and HB. 

\begin{center}
\includegraphics[scale=0.6,trim= 0 180 0 200]{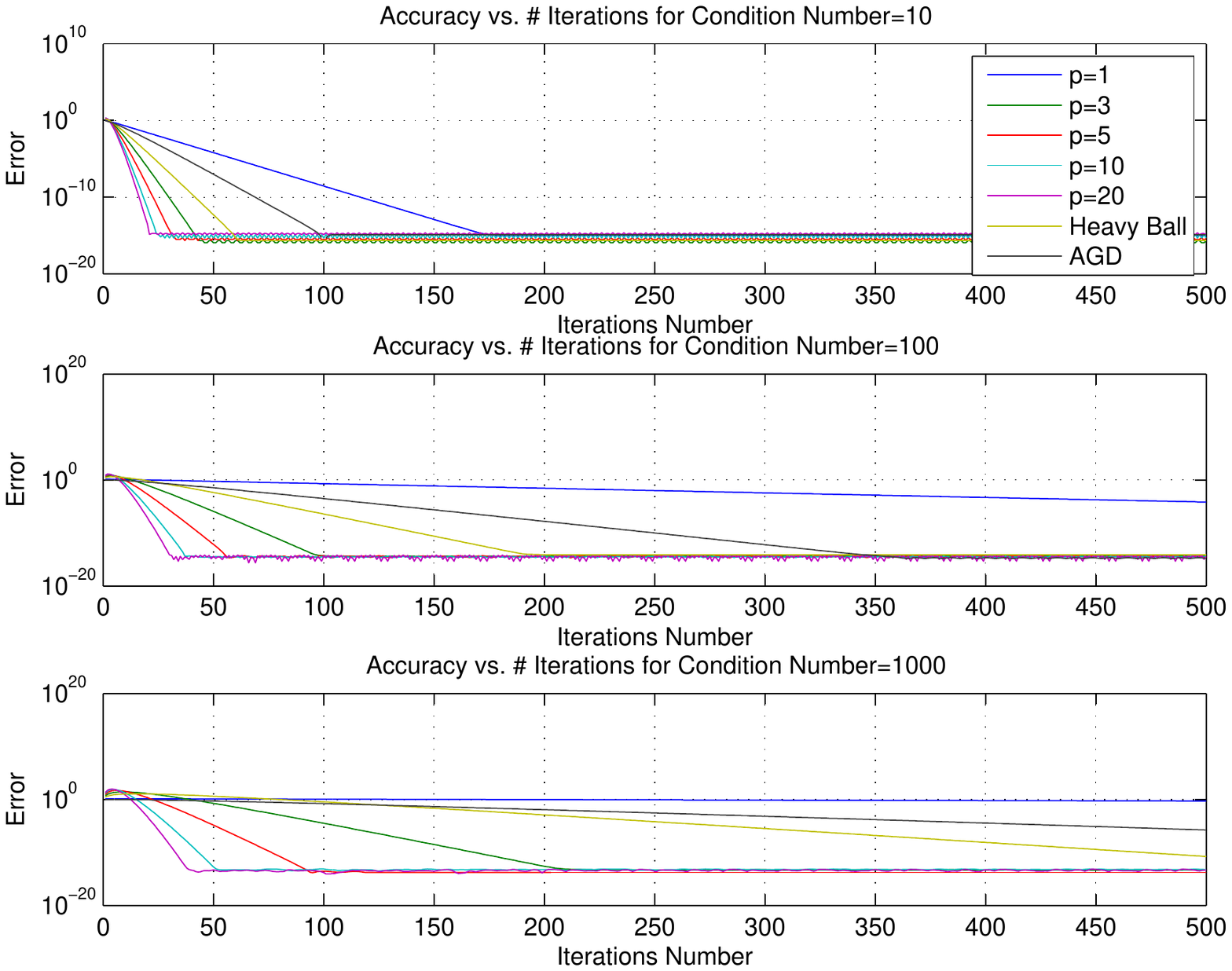}
\end{center}

\section{Linear Coefficient Matrices} \label{section:lin_coeff_mat}
Let $\cM$ be a $p$-CLI method defined by $(M(A),N(A))$. Suppose that $N(A)$ is a scalar matrix denoted by $\nu I_d$, and that $M(A)$ is comprised of linear coefficient matrices $C_0(A),\dots\,C_{p-1}(A)$, i.e., there exist $a_0,\dots,a_{p-1}\in\reals$ and  $b_0,\dots,b_{p-1}\in\reals$ such that
\begin{align*}
C_k(X) = a_k X + b_k I_d ,~\quad k=0,\dots,p-1
\end{align*}

Apart from possessing an efficient implementation, $p$-CLI optimization algorithms with linear coefficient matrices and scalar inversion matrix exhibits another very appealing property, i.e. a canonical extension to first-order optimization algorithms. To see this, recall that by consistency 
\begin{align*}
	\sum_{k=0}^{p-1}C_k(A) =  & I + N(A) A,~ A\in\posdefun{d}{[\mu,L]}
\end{align*} 
Which reduces to 
\begin{align*}
	I_d\sum_{k=0}^{p-1} b_k  +  A\sum_{k=0}^{p-1} a_k   &=   I + \nu A,~ A\in\posdefun{d}{[\mu,L]}
\end{align*} 
Hence,
\begin{align} \label{eq:line_coeff_cons}
	\sum_{k=0}^{p-1} b_k &=1 \text{ and } \sum_{k=0}^{p-1} a_k =\nu
\end{align}
Now, each iteration rule of $\cM$ can be equivalently expressed as
\begin{align*}
	\bx^{t} = C_0(A) \bx^{t-p} + C_1(A) \bx^{t-(p-1)} + \dots +C_{p-1}(A) \bx^{t-1} + \nu \bb
\end{align*} 
Substituting $C_k(A)$ for their definitions yields
\begin{align*}
	\bx^{t} = (a_0 A + b_0) \bx^{t-p} + (a_1 A + b_1)\bx^{t-(p-1)} + \dots +(a_{p-1} A + b_{p-1}) \bx^{p-1} + \nu \bb
\end{align*}
Plugging in \eqref{eq:line_coeff_cons} we get 
\begin{align*}
	\bx^{t} &= (a_0 A + b_0) \bx^{t-p} + (a_1 A + b_1)\bx^{t-(p-1)} + \dots +(a_{p-1} A + b_{p-1}) \bx^{t-1} + \circpar{\sum_{t=0}^{p-1} a_t }\bb\\
	&= 	a_0 (A \bx^{t-p} + \bb) + 	a_1 (A \bx^{t-(p-1)} + \bb) +\dots+ 
		a_{p-1} (A \bx^{t-1} + \bb)\\ &\quad+ 
	 b_0 \bx^{t-p}  + b_1 \bx^{t-(p-1)} + \dots + b_{p-1} \bx^{t-1} 
\end{align*}
By substitute $A\bx + \bb$ for its counterpart $\nabla f(\bx)$ we have
\begin{align} \label{eq:nabla_count}
	\bx^{t} &=  \sum_{k=0}^{p-1} b_k \bx^{t-(p-k)} + \sum_{k=0}^{p-1} a_k \nabla f(\bx^{t-(p-k)})
\end{align}

The performance of the canonical extensions is close to what one obtains for the subclass of quadratic functions and may be analyzed using technique similar to the one used in \cite{polyak1987introduction}. We postpone a more through discussion of this matter to future work.\\

We now present a general scheme for designing $p$-CLI optimization algorithms with linear coefficient matrices and scalar inversion matrix.\\ Denote
\begin{align*}
q(z,\eta)= z^p - (\eta a(z) + b(z))
\end{align*}
According to \thmref{thm:conv_correct}, $\cM$ converges to $-A^{-1}\bb$
for any $A\in\posdefun{d}{[\mu,L]}$ and $\bb\in\reals^d$ if and only for all $\eta\in[\mu,L]$
\begin{align} 
q(1,\eta) &= -\nu\eta  \label{eq:lin_cor1} \\
\spec{q(z,\eta)} &\subseteq (-1,1)  \label{eq:lin_cor2}
\end{align}
In which case, the worst convergence rate of $\cM$ is 
\begin{align} \label{opt:minspec}
\max_{\eta\in[\mu,L]} \rho(q(z,\eta))
\end{align}
These equation encapsulate the two principles by  which one may design an optimal pairing. That is, given $\nu\in\reals$ how should one choose $a_k$ and $b_k$ so that \ref{opt:minspec} is minimized, while maintaining conditions \ref{eq:lin_cor1} and \ref{eq:lin_cor2}. In the following sections we show how to solve (\ref{opt:minspec}) for $p=1$ and $p=2$. We shall refer to two different values of $\nu$:
\begin{align}\label{eq:nu1}
\nu=\frac{-1}{L}
\end{align}
 and 
\begin{align}\label{eq:nu2}
\nu=-\circpar{\frac2{\sqrt[p]{ L} +\sqrt[p]{\mu} }}^p
\end{align}
These are the optimal values for the ranges $[-1/L,0)$ and $(-1/\mu,-1/L)$, respectively, as analyzed in Section \ref{section_non_dep}.\\

\section{Lifting Factor 1 - FGD}
Assume $p=1$. With the notation of the previous section we have,
\begin{align*}
q(z,\eta) = z - \eta a _0 - b_0
\end{align*}
For some $a_0,b_0$. In order to satisfy Condition \ref{eq:lin_cor1} for all $\eta\in[\mu,L]$ we have no other choice but to set 
\begin{align*}
a_0 &= \nu\\
b_0 &= 1
\end{align*} 
It can be easily verified that for this choice of $a_,b_0$, condition \ref{eq:lin_cor2} holds for both values of $\nu$: \ref{eq:nu1} and \ref{eq:nu2}. 
Therefore, for $p=1$ the only iteration matrix for which we have convergence is 
$$M(A)=I + \nu A$$
This is exactly the FGD algorithm (see Section \ref{section_spec_algo}), i.e., the corresponding rule of update is,
\begin{align*}
\bx^{k+1}=(I+\nu A) \bx^{k} + \nu \bb
\end{align*}
and its convergence rate \ref{opt:minspec} is bounded from above by
\begin{align*}
\frac{Q}{Q+1}
\end{align*}
for $\nu=-1/L$, and 
\begin{align*}
\frac{Q-1}{Q+1}
\end{align*}
for $\nu=\frac{-2}{\mu+L}$.

\begin{figure}[H]
  \centering
     \includegraphics[width=1\textwidth]{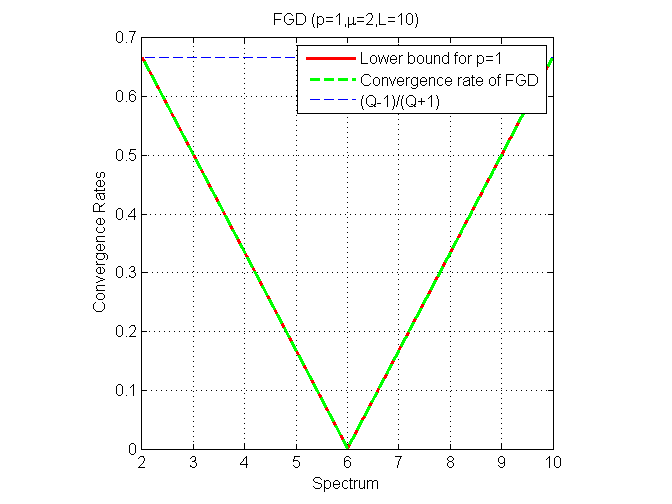}
  \caption{The Convergence rates of Full Gradient Descent vs. spectrum}
\end{figure}

\section{Lifting Factor 2 - AGD and Heavy-Ball}
Assume $p=2$. With the notation of Section \ref{section:lin_coeff_mat}
we have,
\begin{align*}
q(z,\eta)= z^2 - \eta(a_1 z + a_0 ) - (b_1  z + b_0)
\end{align*}
Recall that our goal is to choose $a_1,a_0,b_1,b_0$ so as to minimize 
\ref{opt:comp_poly_min}, while preserving conditions \ref{eq:lin_cor1} and \ref{eq:lin_cor2}. This time $q(z,\eta)$, seen as a function of $\eta$, forms a path of quadratic polynomials. Thus, a natural way to achieve this goal is by choosing the parameters so that $q(z,\mu)$ and $q(z,L)$ will take the form of 'economic' polynomials introduced in Section \ref{section:eco_poly},namely
\begin{align*}
\circpar{z-(1-\sqrt{r})}^2
\end{align*}
for  $r=-\nu\mu$ and $r=-\nu L$, respectively, and hope that for any other $\eta\in(\mu,L)$, the roots of $q(z,\eta)$ would have small moduli. 
This yields the following two equations 
\begin{align*}
q(z,\mu) &= \circpar{z-(1-\sqrt{-\nu\mu)}}^2\\
q(z,L) &= \circpar{z-(1-\sqrt{-\nu L)}}^2
\end{align*}
Plugging in the definition of $q(z,\eta)$ and expanding the r.h.s of these equations above we get,
\begin{align*}
z^2 - ( a_1\mu +b_1)z - (a_0\mu + b_0) &= z^2 -2(1-\sqrt{-\nu\mu})z + (1-\sqrt{-\nu\mu})^2 \\
z^2 - ( a_1  L +b_1)z - (a_0 L+ b_0) &=  z^2 -2(1-\sqrt{-\nu L})z + (1-\sqrt{-\nu L})^2 
\end{align*}
Equivalently, we have the following  system of linear equations,
\begin{align}
- ( a_1\mu +b_1) &= -2(1-\sqrt{-\nu\mu}) \label{eq_params_1}\\
- (a_0 \mu+ b_0) &= (1-\sqrt{-\nu \mu})^2 \label{eq_params_2} \\
- ( a_1 L  +b_1) &= -2(1-\sqrt{-\nu L}) \label{eq_params_3}\\
- (a_0 L+ b_0) &= (1-\sqrt{-\nu L})^2  \label{eq_params_4}
\end{align}
Now, multiply \eqref{eq_params_1} by -1 and add to it \eqref{eq_params_3}. Next, multiply \eqref{eq_params_2} by -1 and add to it \eqref{eq_params_4},
\begin{align*}
 a_1(\mu-L) &= 2\sqrt{-\nu}(\sqrt{L}-\sqrt{\mu})  \\
a_0 (\mu-L) &= (1-\sqrt{-\nu L})^2  -(1-\sqrt{-\nu \mu})^2 
\end{align*}
Thus,
\begin{align*}
a_1 &= \frac{2\sqrt{-\nu}(\sqrt{L}-\sqrt{\mu}) }{\mu-L} 
= \frac{2\sqrt{-\nu}(\sqrt{L}-\sqrt{\mu}) }{(\sqrt{\mu}+\sqrt{L})(\sqrt{\mu}-\sqrt{L})}= \frac{-2\sqrt{-\nu} }{\sqrt{\mu}+\sqrt{L}}\\
a_0  &= \frac{(1-\sqrt{-\nu L})^2  -(1-\sqrt{-\nu \mu})^2 }{\mu-L}
= \frac{(\sqrt{-\nu L} - \sqrt{-\nu \mu})(-2 +\sqrt{-\nu L } +\sqrt{-\nu \mu})}{(\sqrt{\mu}+\sqrt{L})(\sqrt{\mu}-\sqrt{L})}\\
&= \frac{\sqrt{-\nu}(\sqrt{L} - \sqrt{\mu})(-2 +\sqrt{-\nu}(\sqrt{ L } +\sqrt{ \mu}))}
{(\sqrt{\mu}+\sqrt{L})(\sqrt{\mu}-\sqrt{L})}\\
&= \frac{2\sqrt{-\nu} +\nu(\sqrt{ L } +\sqrt{ \mu}))}
{\sqrt{\mu}+\sqrt{L}} = \frac{2\sqrt{-\nu}} {\sqrt{\mu}+\sqrt{L}} + \nu
\end{align*}
Remarkably enough, plugging in $\nu=-1/L$  into the equation above and extracting $b_1,b_0$ accordingly,  yields the exact parameters of AGD ((see Section \ref{section_spec_algo})).  The resulting convergence rate vs. various eigenvalues is,
\begin{figure}[H]
  \centering
     \includegraphics[width=0.8\textwidth]{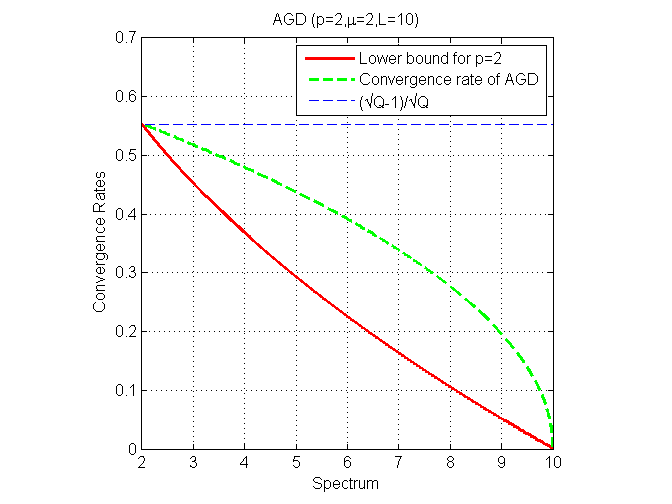}
  \caption{The Convergence rates of the Accelerated Gradient Descent method vs. spectrum}
\end{figure}
Furthermore, by applying the same process with $$\nu = -\circpar{\frac2{\sqrt{L}+\sqrt{\mu}}}^2$$ we get the parameters of Heavy-Ball.
\begin{figure}[H] \label{fig:HB}
  \centering
     \includegraphics[width=0.8\textwidth]{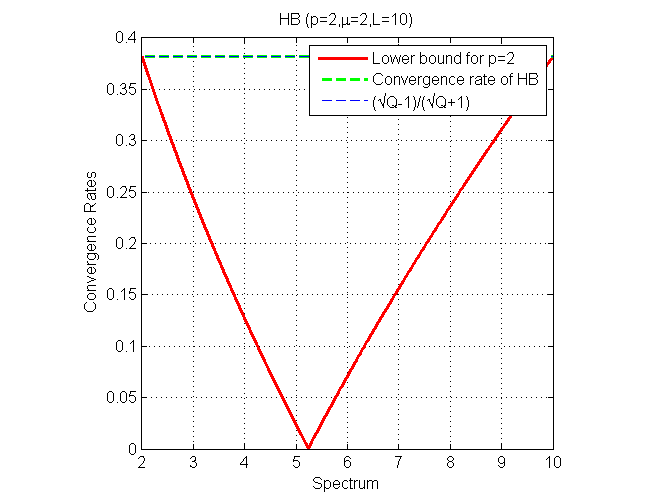}
  \caption{The Convergence rates of the Heavy Ball method vs. spectrum}
\end{figure}

Lastly, it can be seen that condition \ref{eq:lin_cor1} holds for any choice of $\nu$. As for condition \ref{eq:lin_cor2}: Clearly, fitting $q(z,\eta)$ at the edges of the interval $[\mu,L]$ ensures that for $\eta=\mu,L$
\begin{align*}
\rho\circpar{q(z,\eta)}\leq \frac{\sqrt{Q}}{\sqrt{Q}+1}
\end{align*}
for AGD, and 
\begin{align*}
\rho\circpar{q(z,\eta)}\leq \frac{\sqrt{Q}-1}{\sqrt{Q}+1}
\end{align*}
for the Heavy Ball method.

How can one prove that the spectral radius for the intermediate values $(\mu,L)$ does not exceed these bounds at the edges of $[\mu,L]$?  This can be done by re-expressing  $q(z,\eta)$ as a convex combination of $q(z,\mu)$ and $q(z,L)$.\\ For any $\eta\in[\mu,L]$, define $$\theta(\eta)\eqdef\frac{L-\eta}{L-\mu}\in[0,1]$$
It is now straightforward to verify that
\begin{align*}
q(z,\eta) &= \theta(\eta) q(z,\mu)+(1-\theta(\eta)) q(z,L)
\end{align*}
The claim now follows by applying Theorem 4 from \cite{fell1980zeros}.

\section{Lifting Factor \texorpdfstring{$\ge$}{>=} 3}
The conjecture stated in the end of Section \ref{section:is_this_tight} asserts that for any $p$-CLI optimization algorithm $\circpar{p,M(A),N(A)}$, with diagonal inversion matrix and linear coefficient matrices there exists some $A\in\posdefun{d}{[\mu,L]}$ such that 
\begin{align} \label{ineq:sqrt2_lb}
	\rho(M(A))&\ge \frac{\sqrt{Q} - 1}{\sqrt{Q} + 1}  
\end{align}
Thus, it seems hopeless to look for a $p$-CLI optimization algorithm whose iteration and inversion matrices can be computed efficiently and whose convergence rate for any $A\in\posdefun{d}{[\mu,L]}$ is at worst
\begin{align*}
	\frac{\sqrt[p]{Q} - 1}{\sqrt[p]{Q} + 1} 
\end{align*} 
for some $p>2$. \\
That being said, it might be possible to bypass this barrier by focusing on some subclass of $\posdefun{d}{[\mu,L]}$. To this end, recall that the polynomial analogous of this conjecture states that for any $p-1$ degree polynomials $a(z),b(z)$ such that $b(1)=1$ there exists $\eta\in[\mu,L]$ for which 
\begin{align*}
\rho(z^p - (\eta a(z) + b(z))) &\ge \frac{\sqrt{L/\mu} -1}{\sqrt{L/\mu}+1}
\end{align*}
This implies that one might be able to tune $a(z)$ and $b(z)$ so that the convergence rate of the resulting $p$-CLI algorithm on $A\in\posdefun{d}{[\mu,L]}$, whose spectrum does not spread uniformly across $[\mu,L]$, will break \ineqref{ineq:sqrt2_lb}. Note that, as opposed to the $p$-CLI algorithm presented in Section \ref{section:new_algo}, when employing linear coefficient matrices no knowledge regarding the eigenvectors of $A$ is required.   \\

Let us demonstrate this idea for $p=3,\mu = 2$ and $L=100$. Following the exact same derivation used in the last section, let us adjust
\begin{align*}
	q(z,\eta)& \eqdef z^p - (\eta a(z) + b(z))
\end{align*}
so as to have\footnote{As a matter of fact, this regression task reduces to a very simple system of linear equations which can be solved very fast.} 
\begin{align*}
q(z,\mu) &= \circpar{z-(1-\sqrt[3]{-\nu\mu)}}^3\\
q(z,L) &= \circpar{z-(1-\sqrt[3]{-\nu\mu)}}^3
\end{align*}
where 
\begin{align*}
	\nu&=-\circpar{\frac2{\sqrt[3]{ L} +\sqrt[3]{\mu} }}^3
\end{align*}

The resulting 3-CLI optimization algorithm, $\cA_3$, for this choice of parameters is specified as follows
\begin{align*}
	M(A) &\approx \mymat{0_d&I_d&0_d\\0_d&0_d&I_d\\ 0.1958 I_d - 0.0038 A & -0.9850 I_d & 1.7892 I_d - 0.0351 A    } \\
	N(A) &= \nu I_d = -0.0389 I_d
\end{align*}
As illustrated in Figure 5.4, $\cA_3$ obtains a convergence rate of  
\begin{align*}
	\rho_(M_{\cA_3}(A)) \le \frac{\sqrt[3]{Q} - 1}{\sqrt[3]{Q} + 1} 
\end{align*} 
for any $A\in\posdefun{d}{[2,100]}$ which satisfies 
\begin{align*}
	\spec{A}&\subseteq \underbrace{[2,2+\epsilon]\cup[100-\epsilon,100]}_{\hat{\Sigma}},\quad\epsilon\approx 1.5
\end{align*}
thereby outperforms AGD for this family of quadratic functions.

\begin{figure}[H] \label{fig:A_3}
  \centering
	\includegraphics[scale=0.6,trim= 0 180 0 200]{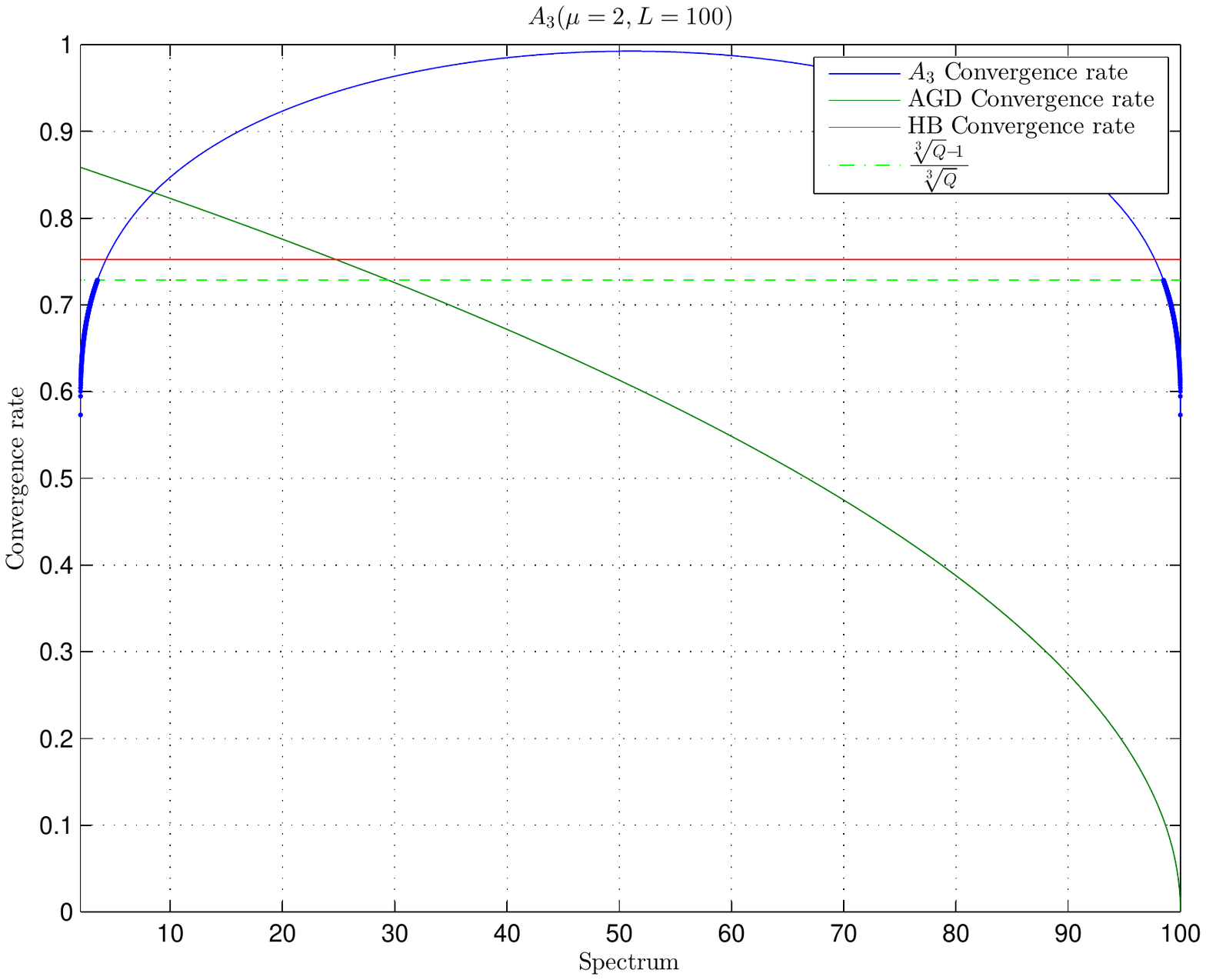}
  \caption{The Convergence rates of $A_3$ 3-CLI optimization algorithm vs. spectrum}
\end{figure}

The following forms a very basic demonstration for the gain in the performance one should expect by using $\cA_3$. We set $\mu =2$ and $L=100$ and define $A$ to be $\Diag{\mu,L}$ rotated by $45^\circ$ counter-clockwise, that is
\begin{align*}
	A &= \mu \mymat{\frac{1}{\sqrt{2}}\\\frac{1}{\sqrt{2}} }
	\mymat{\frac{1}{\sqrt{2}}\\\frac{1}{\sqrt{2}} }^\top
	+
	L \mymat{\frac{1}{\sqrt{2}}\\\frac{-1}{\sqrt{2}} }
	\mymat{\frac{1}{\sqrt{2}}\\\frac{-1}{\sqrt{2}} }^\top
	= \mymat{ \frac{\mu+L}{2} & \frac{\mu-L}{2} \\ \frac{\mu-L}{2} & \frac{\mu+L}{2} }
\end{align*}
We further define  
\begin{align*}
	\bb &= -A \mymat{100\\100}
\end{align*}
Note that $(A,\bb)$ is in $\posdefun{2}{\hat{\Sigma}}$ and that its minimizer is simply $\circpar{100,100}^\top$. The error of a given search point $z\in\reals^{pd}$ is again defined as
\begin{align*}
	\norm{-A^{-1}b  -U^\top z  } 
\end{align*}
($U$ is as defined in \ref{def:U_def})\\
Figure 5.5 shows the decrease in the error of $\cA_3$, AGD and HB along the execution of these algorithms when initialized at $\bx^0 =0$. Clearly, the convergence rate of $\cA_3$ is significantly faster. Indeed, it takes $\sim$40 iterations for $\cA_3$ to obtain an error level of $10^{-6}$, $\sim$80 iterations for HB and $\sim$140 iterations for AGD.

\begin{figure}[H] \label{fig:A_3}
\begin{center}
\includegraphics[scale=0.51,trim= 0 180 0 200]{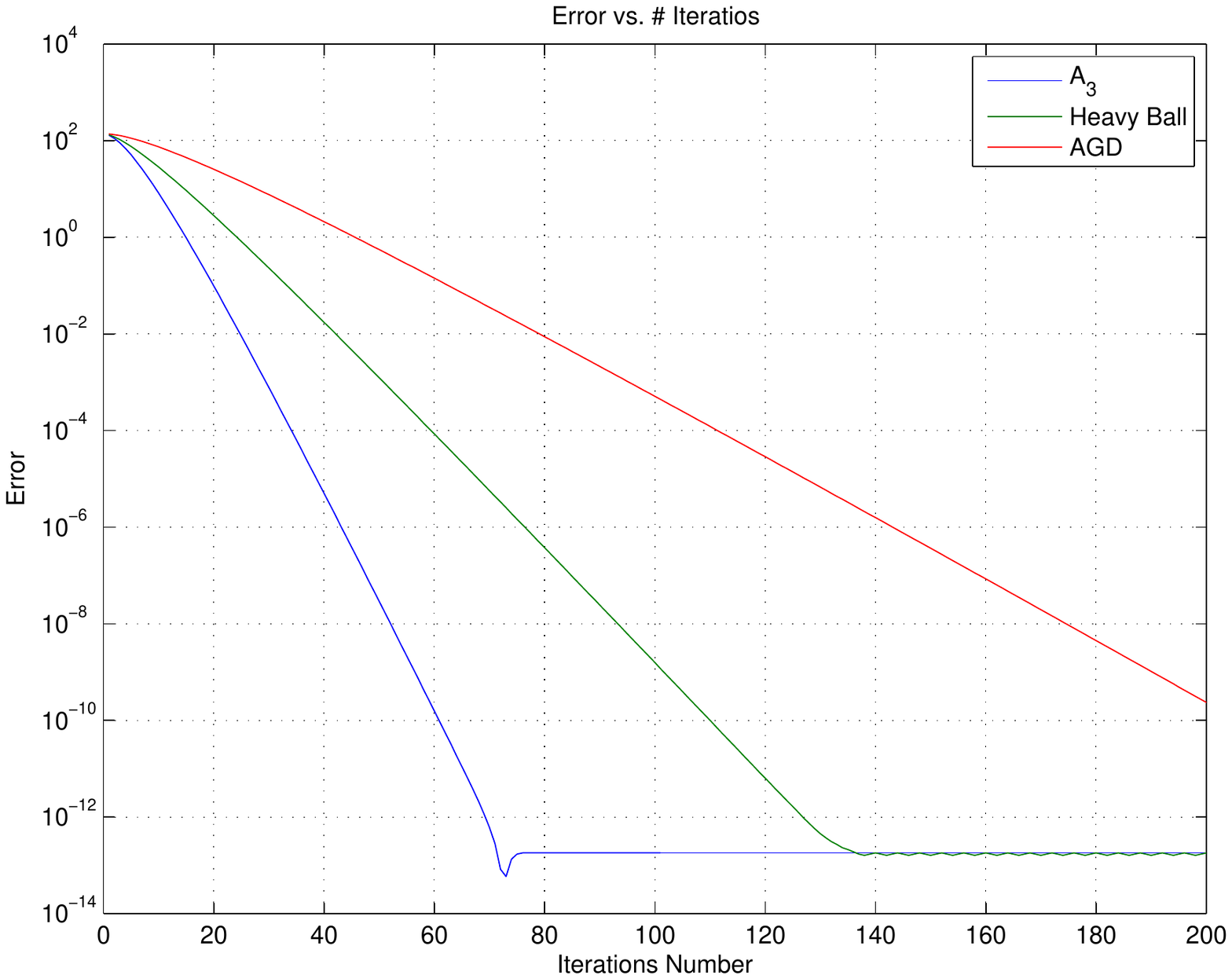}
\end{center}
  \caption{The error of $\cA_3$, AGD and HB as a function of the iteration number when solving a simple quadratic minimization task.}
\end{figure}

It might appear that $\cA_3$ contradicts (\ref{ineq:x_up}), the lower bound which was derived by Nemirovsky and Yudin. Indeed, since $\cA_3$ is a first-order optimization algorithm (see Section \ref{section:lin_coeff_mat}), there must exist some quadratic function $(A_{\text{lb}},\bb_{\text{lb}})\in\posdefun{2}{[\mu,L]}$ for which 
\begin{align*}
	\# \text{ Iterations }  \ge \tilde{\Omega}\circpar{\sqrt{Q}\ln(1/\epsilon)}
\end{align*}
in order to obtain an $\epsilon$-suboptimal solution. Seemingly, this contradicts the upper bound on the iteration complexity of $\cA_3$
\begin{align*}
	\# \text{ Iterations }  \le \tilde{O}\circpar{\sqrt[3]{Q}\ln(1/\epsilon)}  
\end{align*}  

This issue is readily settled by noticing that the upper bound on the iteration complexity of $\cA_3$ holds only for the smaller class of quadratic functions $\posdefun{2}{\hat{\Sigma}}$.
This, in turn, implies that 
\begin{align} \label{ref:bla1}
	(A_{\text{lb}},\bb_{\text{lb}} )\in \posdefun{2}{[\mu,L]} \setminus \posdefun{2}{\hat{\Sigma}}
\end{align}
In other words, we expect successive eigenvalues of $A_{\text{lb}}$ to be relatively close so as to avoid large gaps which may be later exploited by techniques of this sort. It is easier to demonstrate this point using a simplified version 
by Nesterov (see \cite{nesterov2004introductory}). For the sake of simplicity, we consider $1$-smooth $0$-strongly convex functions. In which case, we have
\begin{align*}
	A_{\text{lb}} = \frac{1}{4} \mymat{
	2 & -1 & 0 & \dots && &0 \\
	 -1 & 2 & -1 & 0 & \dots && 0  \\
	0 & -1 & 2 & -1 & 0 & \dots & 0  \\\\
	&&&\ddots\\\\
	0&&\dots&0 &-1& 2 & -1  \\	
	0&&&\dots&0 &-1& 2   \\	
	},~\bb_{\text{lb}} = -\mymat{1\\0\\\vdots\\0}
\end{align*}
As demonstrated by Figure 5.6, $\spec{A_{\text{lb}}}$ densely fills $[\mu,L]$\footnote{
Is is instructive to note that $A_{\text{lb}}$ is, in fact, the standard discrete case central approximation of the second derivative whose spectral decomposition is closely connected to the Chebyshev polynomials of 2nd kind.}, thus confirming (\ref{ref:bla1}).

\begin{figure}[H] \label{fig:NestSpec}
\begin{center}
\includegraphics[scale=0.56,trim= 0 180 0 200]{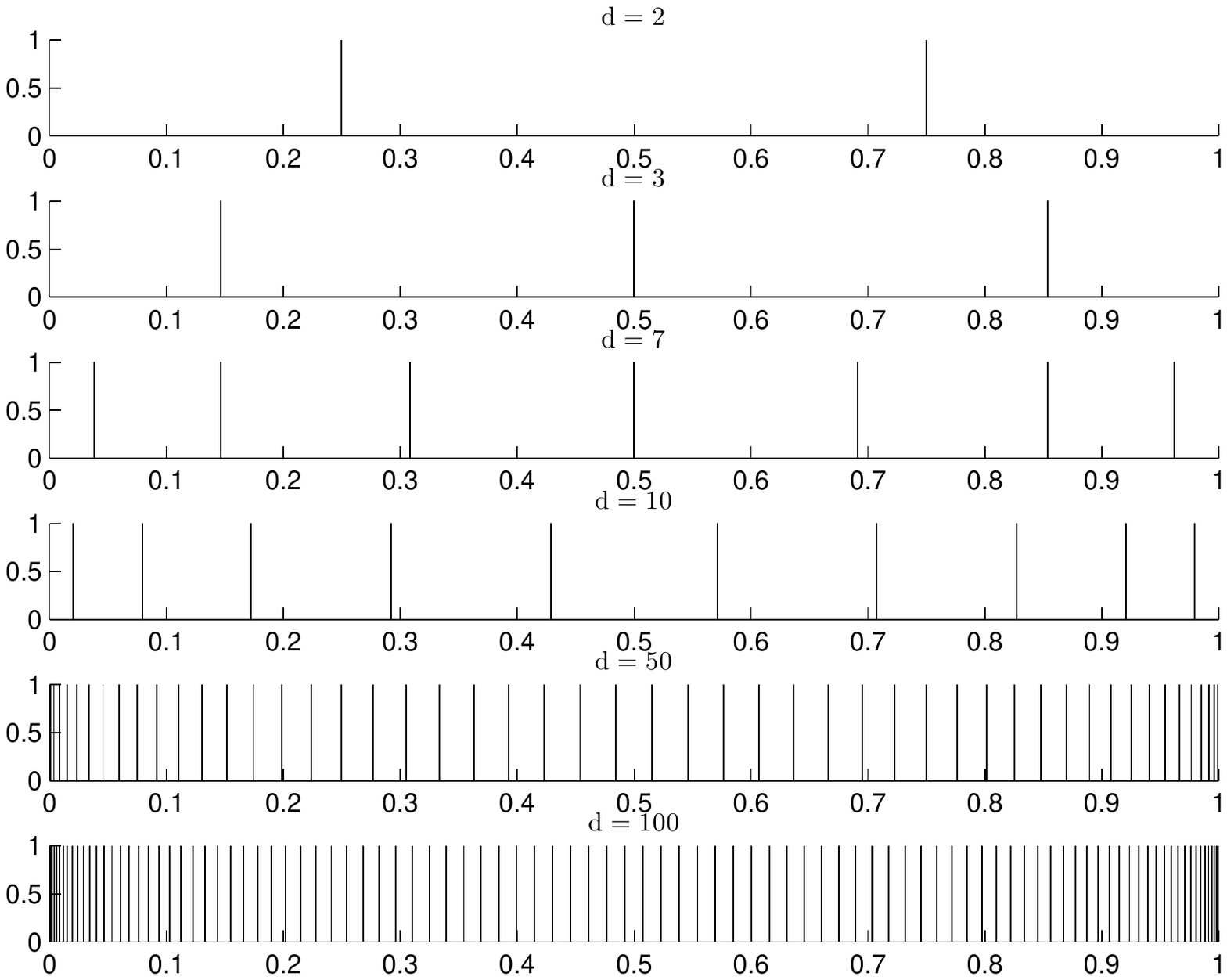}
\end{center}
  \caption{The spectrum of $A_{\text{lb}}$, as used in the derivation of Nesterov's lower bound, for problem space of various dimensions.}
\end{figure}

This technique can may be further generalized to $p>3$ using the same ideas. A different approach would be to use quadratic (or even higher degree) coefficient matrices to exploit other shapes of spectra. In contrast to linear coefficient matrices, it is not clear how one should generalize optimization algorithms resulting from the latter approach to general smooth strongly convex functions. Lastly, the applicability of both approaches heavily depends on how frequent are these types of spectra in real applications.

\chapter{Conclusion and Future Work}
In this monograph we have established a novel framework which forms a bridge between  the world of the optimization algorithms and the world of polynomials, using the theory of linear iterative methods. This framework allowed us to introduce a new perspective on optimization algorithms, derive lower bounds on the convergence rates of existing ones and 're-discover' algorithms which form the cornerstone of  mathematical optimization. We believe that these results are just the tip of the iceberg. Indeed, in what follows we suggest a few research directions for future work.

\begin{description}
\item [Generalizing $p$-CLI Optimization Schemes] Undoubtedly, the most intriguing question is why AGD does generalizes well to general smooth strongly convex functions, and HB does not? How do matrix norms and the spectral radius relate to each other in each case? Answering these questions may pave the way for unified proof of convergence of many optimization algorithms. Outlined schematically, first apply a given optimization algorithm on quadratic functions and then generalize this analysis to any smooth strongly convex function.

\item[Exploiting Spectrum of Various Structures] By carefully designing polynomials, one may be able to gain a significant improvement in the convergence rate of cases where the spectrum is known to obey some structure, such as, large gaps between adjacent eigenvalues. A related question is: Are there cases where obtaining the spectrum of a positive definite matrix is easier than multiplying its inverse by some vector? A positive answer might render the theoretical algorithm presented in Section \ref{section:new_algo} applicable in certain settings.

\item[Constrained Optimization] The essence of this work lies in restricting one's attention to a family of analytically simple functions, i.e., quadratic functions. Is applying the same technique on the constrained counterpart, i.e., quadratic programming, beneficial as well?

\item[Stochastic Setting] Inventing effectively-computed unbiased estimator for the iteration matrix and inversion matrix might give rise for stochastic optimization algorithms, whose steps can be carried out quickly while preserving the original convergence rate.

\item [Expanding the Scope of $p$-CLI Methods] We believe that some assumptions regarding the structure of iteration and inversion matrices can be relaxed so as to broaden the scope of the framework, e.g. to include SAG. Of particular interest is replacing the stationary dynamics assumption with a less restrictive assumption such as periodical dynamics (such as SVRG) or batch dynamics (such as ASDCA). Furthermore, is the simultaneous triangularizablility assumption really necessary? Lastly, 
is there a more intuitive characterization of what does it take for an optimization algorithm to be a $p$-CLI optimization algorithm?

\item[Smoothing Techniques] General convex functions can be made smooth and strongly convex by various modern tools such as the Yosida-Moreau transformation. An instance of a successful implementation of this concept can be found in \cite{guzman2013lower} (For a more general coverage of this topic the reader is referred to \cite{bauschke2011convex}). Do these tools enable us to incorporate optimization algorithms designed for general convex functions in this framework?

\item[Regularized Loss Minimization] By applying the same ideas on quadratic regularized loss minimization (see Section \ref{section:sdca_case_study}), one gets an analogous framework for fundamentals problems in the field of Machine Learning. This may lead to a better understanding of the interplay between the convergence rate and the number of samples of Machine Learning tasks.

\item[Inversion Matrices] In this monograph we consider scalar and diagonal inversion matrices only. What can be said about block diagonal inversion matrices? we believe  that this would not violate the lower bound derived in Chapter \ref{chapter:lower_bounds}. It is likely that this lower bound would withstand any inversion matrix whose entries are weakly dependent on the quadratic function under consideration.

\item[Distributed Computation] It is unclear whether the results presented in this work shed some light as to how one can execute optimization algorithms in distributed systems. For instance, a matrix-vector multiplication can be efficiently computed in parallel. Does this lead to a possible faster execution of $p$-CLI optimization algorithms?

\item[Nonconvex Functions] Some parts of this work are applicable to general, possibly not convex, functions. Is it possible to adapt other parts of this work as well?
\end{description}

\newpage
\appendix 

\chapter{Background}

\section{Linear Algebra} \label{section:app_linear_algebra}
A matrix $A\in\reals^{d\times d}$ is called \emph{Symmetric} if $A_{i,j}=A_{j,i}$, holds for all $i,j\in[d]$. A symmetric matrix $P\in\reals^{d\times d}$ is \emph{Positive Semi definite}, symbolized  $P\succeq0$, if 
\begin{align*} 
x^* P x \ge0 \quad \text{for all nonzero } x\in \bC^d
\end{align*}
and it is \emph{Positive Definite} if 
\begin{align*}
x^* P x >0 \quad \text{for all nonzero } x\in \bC^d
\end{align*}
which we denote by $P\succ 0$.\\
We denote the linear subspace of all symmetric matrices by $\cS^d\subseteq\reals^{d \times d}$. The following relation turns $\cS^d$ into a partial order set (\emph{poset}): 
\begin{align*}
\forall A,B\in\cS^d,\quad B\preceq A \iff  0\preceq A-B
\end{align*}
Note that, if $A$ is a symmetric matrix and $\mu\le L$ are real scalars, then it can be easily verified that 
\begin{align*} 
\mu I_n \preceq A \preceq L I_n 
\end{align*} 
holds if and only if $\spec{A}\subseteq [\mu,L]$. If $\mu>0$ then we may assign a \emph{condition number} $Q\eqdef L/\mu$, for each such matrix. As 

\section{Convex Analysis} \label{section:convex_ana}
The importance of the notion convexity in various fields of Math, Computer Science, Economics, Physics and many others, can not be overstated. That being said, due to limited scope of this work we will be able to take a very brief tour into the theory of Convex Analysis.  Our goal is to establish elementary terminology as well as results which are necessary for our line of inquiry. For a more thorough coverage of this topic see e.g. \cite{nesterov2004introductory,boyd2009convex}.\\

Let $f:\reals^d\to \reals$ be a real-valued function. We say that $f$ is \emph{Convex} if for any $\bx,\by\in \reals^d$ and $\theta\in [0,1]$, 
\begin{align*}
f\circpar{(1-\theta)\bx + \theta \by} \le 
(1-\theta)f\circpar{\bx} +
\theta f\circpar{\by} 
\end{align*}
If there exists $\mu>0$ such that,
\begin{align*}
f\circpar{(1-\theta)\bx + \theta \by} \le 
(1-\theta)f\circpar{\bx} +
\theta f\circpar{\by} -\frac{1}{2}\mu\theta(1-\theta)\normsq{\bx-\by}
\end{align*}
where $\bx,\by\in \reals^d$ and $\theta\in [0,1]$, then we say that $f$ is \emph{$\mu$-Strongly Convex}.
\\
\\
We say that $f$ is \emph{$L$-Smooth} if $f$ is continuously differentiable and its gradient $\nabla{f}$ is $L$-Lipschitz, that is 
\begin{align*}
\norm{\nabla f(\bx)-\nabla f(\by)} \le L\norm{\bx-\by}
\end{align*}
In this monograph we mainly focus on twice continuously  differentiable $L$-smooth $\mu$-strongly convex functions. The following Lemma, which will be stated without proof, is a useful characterization of these kind of functions.
\begin{lemma} \label{ineq_conv_hessian}
Suppose $f:\reals^d\to\reals$ is twice continuously differentiable. Then, $f$ is $L$-smooth $\mu$-strongly convex function if and only if
\begin{align*}
\mu I_d \preceq  \nabla^2 f (\bx) \preceq LI_d\quad x\in\reals^d
\end{align*}
\end{lemma}
For each $L$-smooth $\mu$-strongly convex function we assign a \emph{Condition Number}, defined by $Q=L/\mu$. Note that we used the same term 'condition number' for both matrices and functions. The reason for it is a key connection between positive definite matrices and smooth strongly convex functions:\\
Let us define $f:\reals^d\to\reals$ by
\begin{align} \label{eq:quad_func}
f(x) = \bx^\top A\bx + \bb^\top \bx
\end{align}
where $A\in\cS^d$ is a positive definite matrix (see appendix \ref{section:app_linear_algebra}) and  $\bb\in\reals^d$.\\
One may verify that $f$ is convex if and only if $A\succeq 0$. In addition, $f$ is strongly convex if and only if $A\succ 0$, in which case the condition number of $f$ and $A$ coincides.
Perhaps surprisingly, we shall soon demonstrate the essential role of the condition number, a single scalar associated with a smooth strongly convex function, in characterizing the convergence rate of optimization 
processes.

\chapter{Technical Lemmas} \label{chapter:tech}

\section{Stochastic Average Gradient (SAG)} \label{section:SAG}
This presentation of SAG is due to (\cite{roux2012stochastic}). We follow the notation used in the paper. \\
Let $g:\reals^d\to\reals$ be defined by, 
\begin{align*}
g(x) \eqdef \frac{1}{n}\sum_{i=1}^n f_i(\bx)
\end{align*}
where $f_i$ is $L$-smooth convex function and $g$ is $\mu$-strongly convex.\\
The update rule is defined by 
\begin{align} \label{SAC_x_step}
\bx^{k} = \bx^{k-1} - \frac{\alpha}{n}\sum_{i=1}^n \by^{k}_i
\end{align}
where in the $k$'th iteration we select a random function according to $i_k \sim \cU([n])$ and set,
\begin{align} \label{SAC_y_step}
\by^{k}_i = 
\begin{cases}
\nabla f_{i}(\bx^{k-1}) & \text{if } i=i_k\\
\by^{k-1}_i & \text{o.w.}
\end{cases}
\end{align}

We would like to recast this algorithm as $p$-CLI optimization algorithm for some $p$. However, tackling this problem directly would result in a quite involved expression. In order to partially bypass this issue we apply two tricks that are used in the paper. First, instead of randomly selecting $i_k \sim \cU([n])$, we define a random variable $z_i^k$ by,
\begin{align*}
z_i^k = 
\begin{cases}
1-\frac{1}{n}& \text{w.p. } \frac{1}{n}\\
-\frac{1}{n}& \text{w.p. } 1-\frac{1}{n}\\
\end{cases}
\end{align*}
$\bz^k$ is defined in such way that in each round exactly one $z^k_i$ takes the value $1-\frac{1}{n}$.\\
Thus, $\by^k$ and $\bx^k$ may be equivalently rewritten  as
\begin{align*}
\by^k_i&= \circpar{1-\frac{1}{n}}\by_i^{k-1} + \frac{1}{n} \nabla f_i(\bx^{k-1}) + z_i^k\circpar{\nabla f_i(\bx^{k-1})-\by^{k-1}_i}\\
\bx^k &= \bx^{k-1} - \frac{\alpha}{n} \sum_{i=1}^n \by_i^k \\
&= \bx^{k-1} - \frac{\alpha}{n}\sum_{i=1}^n\left[ \circpar{1-\frac{1}{n}}\by_i^{k-1} + \frac{1}{n} \nabla f_i(\bx^{k-1}) + z_i^k\circpar{\nabla f_i(\bx^{k-1})-\by^{k-1}_i} \right]\\
&= \bx^{k-1} - \frac{\alpha}{n}\sum_{i=1}^n \circpar{1-\frac{1}{n}-z^k_i}\by_i^{k-1} - \frac{\alpha}{n} \sum_{i=1}^n \circpar{\frac{1}{n} + z_i^k} \nabla f_i(\bx^{k-1})  \\
\end{align*}

It is straightforward to verify that probabilistic outcome of the these expressions is equivalent to (\ref{SAC_x_step}), (\ref{SAC_y_step}).\\
Our next task is to convert into operator form. When $f_i$ are quadratic functions, i.e. $f_i(x)=\frac{1}{2}x^\top A_i x,~A\succ0$ then the update rule becomes a linear transformation of the previous test point. This, in turn, enables us to encapsulate all the changes in $\by^k$ and $\bx^k$ in a single step. To this end, let us denote
\begin{align*}
\theta^k &= 
\circpar{
\begin{array} {ccc}
y^k_1\\
y^k_2\\
\vdots\\
y^k_n\\
x^k
\end{array}
},~
&\theta^* &= 
\circpar{
\begin{array} {ccc}
f'_1(x^*)\\
f'_2(x^*)\\
\vdots\\
f'_n(x^*)\\
x^*
\end{array}
}
\end{align*}
\eqref{SAC_x_step} can now be compactly expressed as $\theta^k = \Phi_{z^k} \theta^{k-1}$, where $\Phi_{z^k}$ is defined as,
\begin{align*} 
\circpar{\begin{array}{cccccccccc}
\circpar{1-\frac{1}{n} - z^k_1}I & 0 & \dots && 0 & \circpar{\frac{1}{n}+ z^k_1}A_1 \\
0 & \circpar{1-\frac{1}{n} - z^k_2}I & 0 & \dots & 0 & \circpar{\frac{1}{n}+ z^k_2}A_2 \\
\\
&&\ddots \\\\
0 & \dots & 0 && \circpar{1-\frac{1}{n} - z^k_n}I & \circpar{\frac{1}{n}+ z^k_n}A_n \\
-\frac{\alpha}{n}\circpar{1-\frac{1}{n}-z^k_1 }I &\dots &&&
-\frac{\alpha}{n}\circpar{1-\frac{1}{n}-z^k_n }I &
\circpar{I - \frac{\alpha}{n}\sum_{i=1}^n (\frac{1}{n}+z^k_i) A_i } 
\end{array}}
\end{align*}
Note that, $$\bE [z_i^k] = \frac{1}{n}\circpar{1-\frac{1}{n}} + \circpar{1-\frac{1}{n}}\circpar{\frac{-1}{n}}=0  $$ 
Thus,
\begin{align*} 
\bE[\Phi_z]&=\circpar{\begin{array}{cccccccccc}
\circpar{1-\frac{1}{n}}I & 0 & \dots && 0 & \frac{1}{n}A_1 \\
0 & \circpar{1-\frac{1}{n}}I & 0 & \dots & 0 & \frac{1}{n}A_2 \\
\\
&&\ddots \\\\
0 & \dots & 0 && \circpar{1-\frac{1}{n}}I & \frac{1}{n}A_n \\
-\frac{\alpha}{n}\circpar{1-\frac{1}{n} }I &\dots &&&
-\frac{\alpha}{n}\circpar{1-\frac{1}{n}}I &
I - \frac{\alpha}{n}\sum_{i=1}^n \frac{1}{n} A_i  
\end{array}}
\end{align*}
Which is closely related to the iteration matrix of general $d(n+1)$-CLI methods.

\section{Convergence Rates and Iteration Complexity}

For any sequence of positive real scalars $\set{a_k}_{k=1}^\infty$, we define $K:\reals^{++}\to\bN$ by
\begin{align*}
K(\epsilon) = \min \myset{k\in \bN }{a_k<\epsilon}
\end{align*}
The following lemma considers the asymptotic behavior of $K(\epsilon)$ for certain types of sequences.

\begin{lemma} \label{lemma:spec_to_rate}
Let $\set{a_k}_{k=1}^\infty$ be a sequence of real scalars, for which there exist $\rho_1\in(0,1)$, $c,C$ positive scalars and $m\in\bN$ such that
\begin{align}
ck^m\rho^k \le a_k \le Ck^m\rho^k
\end{align}
for sufficiently large $k$. 
Then 
\begin{align*}
K(\epsilon)=\Omega\circpar{\frac{\rho}{1-\rho}\ln(1/\epsilon)}
\end{align*}
and
\begin{align*}
K(\epsilon)=\bigO{\frac{1}{1-\rho}\ln(1/\epsilon)}
\end{align*}
Thus, since $\rho\in(0,1)$ we get
\begin{align} \label{eq:conv_rate_to_ic}
K(\epsilon)=\Theta\circpar{\frac{\rho}{1-\rho}\ln(1/\epsilon)}
\end{align}
\end{lemma}
\begin{proof}
The following inequality,
\begin{align*}
1-\frac{2}{x+1}\ge \exp\circpar{\frac{-2}{x-1}}
\end{align*}
which applies for any $x\ge 1$, yields 
\begin{align} \label{ineq:rho_exp}
x &= 1 - (1-x) = 1 - \frac{1}{\frac{1}{(1-x)}} = 1 - \frac{2}{\frac{2}{(1-x)}} = 1 - \frac{2}{\circpar{\frac{2}{(1-x)}-1} + 1}\nonumber\\
&\ge \exp\circpar{\frac{-2}{\circpar{\frac{2}{(1-x)}-1} - 1}} = \exp\circpar{\frac{-2}{\frac{2}{(1-x)}- 2}} = \exp\circpar{\frac{-1}{\frac{1}{(1-x)}- 1}}
\end{align}
Thus,
\begin{align*}
\epsilon>a_k\ge ck^m\rho^k \ge c\rho^k  \ge c\exp\circpar{\frac{-k}{\frac{1}{(1-\rho_1)}- 1}}  
\end{align*}
(We might as well assume that the last inequality holds for any $k$).\\
Taking natural logarithm from both sides and rearranging we have,
\begin{align*}
\frac{\rho}{1-\rho}\circpar{\ln(1/\epsilon) + \ln(c)}< k  \\
\end{align*}

The upper bound is proven in a similar way using the following inequality,
\begin{align*}
x\le \exp(x-1)
\end{align*}
\end{proof}

\bibliographystyle{alpha}
\bibliography{thesis_bib}

\end{document}